\documentclass{article}
\usepackage[margin=4cm]{geometry}
\usepackage{amsthm, pgfplots, tikz, environ, subcaption, amsmath, amssymb, mathtools, algpseudocode, algorithm, multirow, url, times}
\usetikzlibrary{positioning,matrix,decorations,calc,arrows,decorations.markings,patterns,automata,arrows.meta}
\newcommand{\NP}{$\mathcal{NP}$}
\newcommand{\hide}[1]{}

\newcommand{\doi}[1]{\textsc{doi}: \href{http://dx.doi.org/#1}{\nolinkurl{#1}}}
\newcommand{\arxiv}[1]{\textsc{arXiv}: \href{http://arxiv.org/abs/#1}{\nolinkurl{#1}}}

\newcounter{ssection}

\newcommand{\sslabel}[1]{%
	\label{#1}%
	\refstepcounter{ssection}\label{ssection:#1}%	
}

\newcounter{sssection}

\newcommand{\ssslabel}[1]{%
	\label{#1}%
	\refstepcounter{sssection}\label{sssection:#1}%	
}

\newtheorem{proposition}{\bf Proposition}[section]
\newtheorem{lemma}{\bf Lemma}[section]

\makeatletter
\let\matamp=&
\catcode`\&=13
\makeatletter
\def&{\iftikz@is@matrix
	\pgfmatrixnextcell
	\else
	\matamp
	\fi}
\makeatother

\newcounter{lines}
\def\endlr{\stepcounter{lines}\\}

\newcounter{vtml}
\newif\ifvtimelinetitle
\newif\ifvtimebottomline
\tikzset{description/.style={
		column 2/.append style={#1}
	},
	timeline color/.store in=\vtmlcolor,
	timeline color=red!80!black,
	timeline color st/.style={fill=\vtmlcolor,draw=\vtmlcolor},
	use timeline header/.is if=vtimelinetitle,
	use timeline header=false,
	add bottom line/.is if=vtimebottomline,
	add bottom line=false,
	timeline title/.store in=\vtimelinetitle,
	timeline title={},
	line offset/.store in=\lineoffset,
	line offset=4pt,
}

\NewEnviron{vtimeline}[1][]{%
	\setcounter{lines}{1}%
	\stepcounter{vtml}%
	\begin{tikzpicture}[column 1/.style={anchor=east},
	column 2/.style={anchor=west},
	text depth=0pt,text height=1ex,
	row sep=1ex,
	column sep=1em,
	#1
	]
	\matrix(vtimeline\thevtml)[matrix of nodes]{\BODY};
	\pgfmathtruncatemacro\endmtx{\thelines-1}
	\path[timeline color st] 
	($(vtimeline\thevtml-1-1.north east)!0.5!(vtimeline\thevtml-1-2.north west)$)--
	($(vtimeline\thevtml-\endmtx-1.south east)!0.5!(vtimeline\thevtml-\endmtx-2.south west)$);
	\foreach \x in {1,...,\endmtx}{
		\node[circle,timeline color st, inner sep=0.15pt, draw=white, thick] 
		(vtimeline\thevtml-c-\x) at 
		($(vtimeline\thevtml-\x-1.east)!0.5!(vtimeline\thevtml-\x-2.west)$){};
		\draw[timeline color st](vtimeline\thevtml-c-\x.west)--++(-3pt,0);
	}
	\ifvtimelinetitle%
	\draw[timeline color st]([yshift=\lineoffset]vtimeline\thevtml.north west)--
	([yshift=\lineoffset]vtimeline\thevtml.north east);
	\node[anchor=west,yshift=16pt,font=\large]
	at (vtimeline\thevtml-1-1.north west) 
	{\textsc{Timeline \thevtml}: \textit{\vtimelinetitle}};
	\else%
	\relax%
	\fi%
	\ifvtimebottomline%
	\draw[timeline color st]([yshift=-\lineoffset]vtimeline\thevtml.south west)--
	([yshift=-\lineoffset]vtimeline\thevtml.south east);
	\else%
	\relax%
	\fi%
	\end{tikzpicture}
}

\newcommand{\WangTile}[4]{%
	\ensuremath{\vcenter{\hbox{%
				\begin{tikzpicture}[every node/.style={inner sep=0,outer sep=0}]
				\draw[draw=black,thick] (0,0) rectangle ++(1,1);
				\draw (1,1) node{}
				-- (0,1) node{}
				-- (0.5,0.5) node{}
				-- cycle;
				\draw (0,1) node{}
				-- (0,0) node{}
				-- (0.5,0.5) node{}
				-- cycle;
				\draw (0,0) node{}
				-- (1,0) node{}
				-- (0.5,0.5) node{}
				-- cycle;
				\draw (1,0) node{}
				-- (1,1) node{}
				-- (0.5,0.5) node{}
				-- cycle;
				\draw (0.5,0.80) node{\scriptsize #1};
				\draw (0.20,0.5) node{\scriptsize #2};
				\draw (0.5,0.20) node{\scriptsize #3};
				\draw (0.80,0.5) node{\scriptsize #4};
				\end{tikzpicture}%
}}}}

\newcommand{\WT}[6]{%
	\draw[draw=black,fill=white,thick] ({0+#5},{0+#6}) rectangle ++(1,1);
	\draw ({1+#5},{1+#6}) node{}
	-- ({0+#5},{1+#6}) node{}
	-- ({0.5+#5},{0.5+#6}) node{}
	-- cycle;
	\draw ({0+#5},{1+#6}) node{}
	-- ({0+#5},{0+#6}) node{}
	-- ({0.5+#5},{0.5+#6}) node{}
	-- cycle;
	\draw ({0+#5},{0+#6}) node{}
	-- ({1+#5},{0+#6}) node{}
	-- ({0.5+#5},{0.5+#6}) node{}
	-- cycle;
	\draw ({1+#5},{0+#6}) node{}
	-- ({1+#5},{1+#6}) node{}
	-- ({0.5+#5},{0.5+#6}) node{}
	-- cycle;
	\draw ({0.5+#5},{0.8+#6}) node{\scriptsize #1};
	\draw ({0.2+#5},{0.5+#6}) node{\scriptsize #2};
	\draw ({0.5+#5},{0.2+#6}) node{\scriptsize #3};
	\draw ({0.8+#5},{0.5+#6}) node{\scriptsize #4};
}

\newcommand{\CornerTile}[4]{%
	\ensuremath{\vcenter{\hbox{%
				\begin{tikzpicture}[every node/.style={inner sep=0,outer sep=0}]
				\draw[draw=black, thick] (0,0) rectangle ++(1,1);
				\draw (1,1) node{}
				-- (0.5,1) node{}
				-- (0.5,0.5) node{}
				-- (1, 0.5) node{}
				-- cycle;
				\draw (0,1) node{}
				-- (0,0.5) node {}
				-- (0.5,0.5) node{}
				-- (0.5,1) node{}
				-- cycle;
				\draw (0,0) node{}
				-- (0.5,0) node{}
				-- (0.5,0.5) node{}
				-- (0, 0.5) node{}
				-- cycle;
				\draw (1,0) node{}
				-- (1,0.5) node{}
				-- (0.5,0.5) node{}
				-- (0.5,0) node{}
				-- cycle;
				\draw (0.25,0.75) node{\scriptsize #1};
				\draw (0.25,0.25) node{\scriptsize #2};
				\draw (0.75,0.25) node{\scriptsize #3};
				\draw (0.75,0.75) node{\scriptsize #4};
				\end{tikzpicture}%
}}}}


\date{}

\allowdisplaybreaks

\begin{document}

%%%% Article title to be placed here
\title{On bounded Wang tilings}

\author{%%%% Author details
Marek Tyburec\footnote{Department of Mechanics, Faculty of Civil Engineering, Czech Technical University in Prague, Th\'akurova 7, 16000 Prague 6, Czech Republic} \and Jan Zeman}
\maketitle

%%%%%%%%% Insert author address here
%\address{$^{1}$Department of Mechanics, Faculty of Civil Engineering, Czech Technical University in Prague, Th\'akurova 7, 16000 Prague 6, Czech Republic}

%%%% Subject entries to be placed here %%%%
%\subject{combinatorics, mathematical logic}

%%%% Keyword entries to be placed here %%%%
%\keywords{bounded Wang tiling, maximum-cover tilings, heuristics, integer programming}

%%%% Insert corresponding author and its email address}
%\corres{Marek Tyburec\\
%\email{marek.tyburec@cvut.cz}}

%%%% Abstract text to be placed here %%%%%%%%%%%%
\begin{abstract}
Wang tiles enable efficient pattern compression while avoiding the periodicity in tile distribution via programmable matching rules. However, most research in Wang tilings has considered tiling the infinite plane. Motivated by emerging applications in materials engineering, we consider the bounded version of the tiling problem and offer four integer programming formulations to construct valid or nearly-valid Wang tilings: a~decision, maximum-rectangular tiling, maximum cover, and maximum adjacency constraint satisfaction formulations. To facilitate a~finer control over the resulting tilings, we extend these programs with tile-based, color-based, packing, and variable-sized periodic constraints. Furthermore, we introduce an efficient heuristic algorithm for the maximum-cover variant based on the shortest path search in directed acyclic graphs and derive simple modifications to provide a~$1/2$ approximation guarantee for arbitrary tile sets, and a~$2/3$ guarantee for tile sets with cyclic transducers. Finally, we benchmark the performance of the integer programming formulations and of the heuristic algorithms showing that the heuristics provides very competitive outputs in a~fraction of time. As a by-product, we reveal errors in two well-known aperiodic tile sets: the Knuth tile set contains a~tile unusable in two-way infinite tilings, and the Lagae corner tile set is not aperiodic.
\end{abstract}

\section{Introduction}\label{sec:wangintro}
Wang tiles, non-rotatable unit squares with colored edges, constitute a~formalism introduced by Wang~\cite{wang1963} to visualize the $\forall \exists \forall$ decidability problem of predicate calculus. Formulating an equivalent domino problem, Wang considered an infinite number of copies of an arbitrary set of Wang tiles and investigated whether there exists a~simply-connected valid tiling of the infinite plane. Moreover, he conjectured in \cite{wang1961} that only the tile sets that form a~torus, i.e., cover a~periodic simply-connected rectangular domain with identical coloring of the opposite edges, generate infinite valid tilings. Berger \cite{berger1966} disproved the conjecture by finding a~tile set that covers the infinite plane aperiodically by exploiting Kahr's reduction of the Turing halting problem~\cite{turing1936,davis1958} to the origin-constrained domino problem~\cite{kahr1962}. Hence, the domino problem was proven to be \textit{undecidable} and, consequently, no general finite algorithm for producing infinite valid tilings exists.

Far less attention has been paid to the finite version of the domino problem, \textit{bounded tiling}, i.e., searching for a~fixed-sized valid tiling generated by an arbitrary tile set. Although the problem is known to be \NP-complete in general, e.g., \cite{lewis1978} or~\cite[Theorem 7.2.1]{lewis1981}, most of the available approaches exploit specific properties of particular tile sets, e.g., \cite{cohen2003,derouet2017,lagae2005,ollinger2008}. However, several closely related works address the tile packing problem for edge-matching puzzles, in which all tiles from the set are placed exactly once; see, e.g., \cite{kovalsky2015,lagae2007,yu2015} and \cite{salassa2017} for an approach aiming at the famous Eternity II puzzle.

In this work, we investigate the bounded Wang tiling problem in its full generality. To this goal, we first survey the most significant \textit{aperiodic} tile sets in Section \ref{ssection:sec:aptile} and applications of Wang tiles in Section \ref{ssection:sec:significance}. In Section \ref{ssection:sectilingalg}, we list available algorithms for generation of Wang tilings. Finally, our aims and contributions appear summarized in Section \ref{ssection:sec:wangnovelty}.

\subsection{Aperiodic tile sets}\sslabel{sec:aptile}

\begin{figure}[!b]
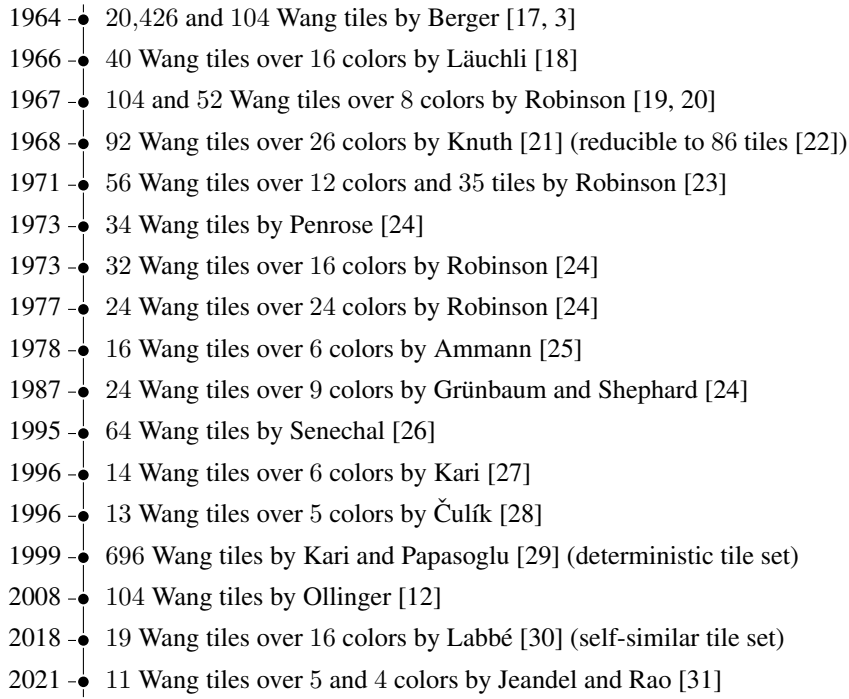

	\vspace*{-7pt}
	\begin{vtimeline}[timeline color=black, line offset=2pt]
		1964 & $20{\small,}426$ and $104$ Wang tiles by Berger \cite{berger1964, berger1966}\endlr
		1966 & $40$ Wang tiles over $16$ colors by L{\"{a}}uchli \cite{wang1975}\endlr
		1967 & $104$ and $52$ Wang tiles over $8$ colors by Robinson \cite{robinson1967,poizat1980}\endlr
		1968 & $92$ Wang tiles over $26$ colors by Knuth \cite{knuth1968} (reducible to $86$ tiles~\cite{knuth2018})\endlr
		1971 & $56$ Wang tiles over $12$ colors and $35$ tiles by Robinson~\cite{robinson1971}\endlr
		1973 & $34$ Wang tiles by Penrose~\cite{grunbaum1987}\endlr
		1973 & $32$ Wang tiles over $16$ colors by Robinson~\cite{grunbaum1987}\endlr
		1977 & $24$ Wang tiles over $24$ colors by Robinson~\cite{grunbaum1987}\endlr
		1978 & $16$ Wang tiles over $6$ colors by Ammann~\cite{robinson1978}\endlr
		1987 & $24$ Wang tiles over $9$ colors by Gr\"unbaum and Shephard~\cite{grunbaum1987}\endlr
		1995 & $64$ Wang tiles by Senechal~\cite{senechal1996}\endlr
		1996 & $14$ Wang tiles over $6$ colors by Kari~\cite{kari1996}\endlr
		1996 & $13$ Wang tiles over $5$ colors by \v{C}ul\'ik \cite{culik1996}\endlr
		1999 & $696$ Wang tiles by Kari and Papasoglu~\cite{kari1999} (deterministic tile set)\endlr
		2008 & $104$ Wang tiles by Ollinger~\cite{ollinger2008}\endlr
		2018 & $19$ Wang tiles over $16$ colors by Labb\'e~\cite{labbe2018} (self-similar tile set) \endlr
		2021 & $11$ Wang tiles over $5$ and $4$ colors by Jeandel and Rao~\cite{jeandel2021}\endlr
	\end{vtimeline}
	\caption{List of aperiodic Wang tile sets.}
	\label{fig:timeline}
	\vspace*{-5pt}
\end{figure}

The originally unexpected property of Wang tile sets---aperiodicity---resulted in a~long-term competition among scientists in mathematical logic, computer science, discrete mathematics, and even recreational mathematicians to find the aperiodic tile set of the minimum cardinality \cite[Chapter 11]{grunbaum1987}. Starting from the Berger tile set containing $20{\small,}426$ tiles in $1964$~\cite{berger1964,berger1966}, it took almost $50$ years until the two sets of $11$ tiles were found and proved to be minimal~\cite{jeandel2021}; see Fig.~\ref{fig:timeline} for a~graphical overview of the selected historical developments.

In 1966, L\"auchli sent to Wang an aperiodic set of $40$~tiles over $16$~colors, but it remained unpublished until 1975~\cite{wang1975}. Meanwhile, unaware of the L\"auchli's result, Knuth \cite{knuth1968} simplified Berger's set to $92$~tiles over $26$~colors; and Robinson developed sets of $104$ and $52$~tiles over $8$ colors in 1967~\cite{poizat1980}, of $56$~tiles over $12$ colors in 1971~\cite{robinson1971}, and anticipated an existence of a~set of $35$~tiles~\cite{robinson1971}.

In 1973, Penrose developed a~new approach based on kite and dart tiling, leading to a~set of $34$ tiles \cite{grunbaum1987}. Robinson, being in contact with Penrose, modified Penrose's approach to reach a~reduced set of $32$ tiles over $16$ colors~\cite{grunbaum1987}. Using the same technique together with Penrose rhombs tiling, Gr\"unbaum and Shephard \cite{grunbaum1987} obtained a~set of $24$ tiles over $9$~colors in 1987.

Another two tile sets were discovered by Ammann. In 1978, he used the so-called Ammann bars to reach $16$~tiles over $6$~colors~\cite{robinson1978}. Building on the Ammann's A2~tiling, see, e.g.,~\cite{grunbaum1987}, Robinson obtained a~set of $24$~tiles over $24$~colors in 1977 \cite{grunbaum1987}.

Subsequent size reduction of the smallest aperiodic set occurred in 1996, when Kari~\cite{kari1996} developed a~new method based on Mealy machines multiplying Beatty sequences and presented a~set of $14$ tiles over $6$ colors. \v{C}ul\'ik~\cite{culik1996}, using the same approach, reduced the set even further to $13$ tiles over $5$ colors.

The search for the minimal aperiodic set was concluded by Jeandel and Rao~\cite{jeandel2021}, who used an enumeration approach to find aperiodic sets of $11$ tiles over $4$ and $5$ colors and proved non-existence of an aperiodic set either containing $10$ or fewer tiles or labeled by less than $4$ colors.

In addition to the original Wang tiles, in 2006, Lagae and Dutr\'{e}~\cite{lagae2006} described a~subset of the Wang tiles, the \textit{corner} tiles (we refer to Appendix \ref{sec:corner_edge} for their relation to edge-based Wang tiles), with the adjacency rules stored in the colored corners instead of the edges. In the same year, they constructed multiple aperiodic sets of corner tiles~\cite{lagae2006report}, out of which the set of $44$ corner tiles over $6$ colors was the smallest one. The set was further simplified by Nurmi~\cite{nurmi2016} to $30$ corner tiles over $6$ colors and both were claimed to be aperiodic.

\subsection{Applications of Wang Tiles}\sslabel{sec:significance}
Thanks to the property of particular tile sets to generate aperiodic tilings, Wang tiles gained interest among several disciplines. Building on the original purpose of Wang tiles, proofs in the first-order logic~\cite{wang1961}, they were also used in cellular automata theory~\cite{kari1990}, topology, group theory \cite{conway1990}, and symbolic dynamical systems~\cite{mozes1989}.

In computer graphics, Stam~\cite{stam1997} adopted Wang tiles to generate aperiodic textures. After Cohen \textit{et al.}~\cite{cohen2003} recognized that stochastic nonperiodic tilings are easier to handle, the Wang-tile-based approach to generating seamless textures became popular, also including the generation of point patterns \cite{cohen2003,hiller2001}.

In science, Wang tiles and other related aperiodic tilings served as the key tool for understanding the $5$-fold symmetry of electron diffraction patterns of quasicrystals~\cite{radin1987,senechal1996}. Another application at the nanoscale involved molecular DNA-based realization of Wang tiles, introduced by~Winfree \textit{et al.}~\cite{winfree1998}, which provided a~self-assembly of biological nanostructures into aperiodic patterns. The self-assembly process of DNA Wang tiles also powered custom DNA-based computations~\cite{seeman2000}, fueled by Turing completeness of Wang tiles~\cite{berger1966,wang1975}.

Beyond the nanoscale, Wang tiles have also been used for efficient compression~\cite{novak2012} and reconstruction~\cite{doskar2014} of nonperiodic microstructures, speeding up numerical analyses of random heterogeneous materials \cite{doskar2016,doskar2020}. Furthermore, we have developed a~bilevel optimization approach to design modular truss materials based on the corner Wang tiling formalism~\cite{tyburec2020} and a~clustering-based method for designing modular structures and mechanisms with continuum topology optimization \cite{tyburec2021arxiv}. Finally, J\'ilek \textit{et al.}~\cite{jilek2021} developed a~centimeter-scale self-assembly procedure building on the Wang tiling formalism.

\subsection{Wang tiling generation algorithms}\sslabel{sectilingalg}

To the best of our knowledge, no general approaches to solving the bounded tiling problem have been reported in the literature; the only available results are specific to single families of tile sets~\cite{cohen2003,derouet2017,lagae2005,ollinger2008}, or consider infinite thin strips~\cite{jeandel2021}. In what follows, we describe the gist of three tiling algorithms: substitution-based, stochastic, and transducer-based.

\paragraph{Substitution-based tiling algorithm}

Given a~tile set $\mathcal{T}$, substitution is a~map $S: \mathcal{T} \mapsto \mathfrak{T}$ that assigns a~tiling $\mathfrak{T}_k$ to each tile $k \in \mathcal{T}$; we refer the reader to Section \ref{sec:definitions} for the definitions. Consequently, arbitrary-sized tilings follow from a placed tile $k$ and a recursively applied substitution rule~\cite{ollinger2008}. Hence, the tiling ``grows'' iteratively. Clearly, such a~procedure has a~low complexity, but only very specific tile sets allow for such substitution rule that generates valid tilings.

\paragraph{Stochastic tiling algorithms}

In computer graphics, Wang tiles have mostly been used for generating visually appealing yet compressed textures. For this, it is essential to generate these nonperiodic patterns quickly, which is best achieved with stochastic tile sets---usually containing all combinations of edge labels for a~given number of colors. For example, in the stochastic tiling algorithm~\cite{cohen2003}, the tiling is generated row-wise, such that the neighbor of any tile that has already been placed can always be selected from at least two tiles at random. This approach was further extended towards the hash-based direct stochastic tiling algorithm~\cite{lagae2005}. Note that stochastic algorithms enable straightforward enforcement of several tile- or edge-based constraints.

\paragraph{Transducer-based tiling algorithm}

The transducer-based tiling algorithm~\cite{jeandel2021} builds on the fact that the $1$D domino problem is decidable and can be solved in a~polynomial time because the bi-infinite path is formed by an arbitrary cycle in transducer graphs, see Section~\ref{sec:definitions} for clarification. To generate valid tilings of multiple rows, the product of several transducers must be computed. Hence, we must enumerate all feasible valid tilings for a~requested height and unit width, and then find a~path of the given length in the transducer graph of the just-formed tile set. Obviously, this approach works well for tiling thin strips; however, it is impractical for larger nearly-square domains.

\subsection{Aims and novelty}\sslabel{sec:wangnovelty}

In this contribution, we consider the bounded Wang tiling in its general form, thereby allowing arbitrary tile sets and control over the resulting tilings. As follows from the above state-of-the-art survey, no such method has been published yet.

We believe that development of such algorithms is important from multiple reasons. First, we have already investigated modeling and optimization of non-periodic and stochastic microstructures with Wang tilings, e.g., \cite{novak2012,doskar2014,doskar2016,doskar2018,doskar2020,tyburec2020}. We hope that the extension of our methods to more general tile sets would enable characterizing a~broader class of non-periodic conventional materials and meta-materials \cite{coulais2016,yang2018,nezerka2018} and thus improve upon the performance of optimized designs.

Apart from emerging applications in materials engineering, we believe that developing a~unified methodology is of independent interest, e.g., for the verification of the results available in the literature. Here we justify this claim by finding two errors in well-established aperiodic tile sets.

To do this, we first provide the necessary definitions in Section \ref{sec:definitions} to make the manuscript self-contained. The subsequent part is devoted to four integer programming formulations for generation of valid tilings: decision variant in Section \ref{ssection:sec:feasibility}, maximum rectangular valid tiling in Section \ref{ssection:sec:maxrect}, maximum-cover in Section \ref{ssection:sec:maxcov}, and maximum adjacency constraint satisfaction in Section \ref{ssection:sec:maxadjacency}. To allow for a~finer control over the resulting tilings, we also include simple extensions to prescribe tile- and color-based boundary conditions, periodic constraints, and the tile-packing constraint in Section \ref{ssection:sec:extensions}.

Due to the complexity of the proposed formulations, in Section \ref{sec:heuristics} we propose a~heuristic graph-based algorithm to tackle the maximum-cover optimization variant from Section \ref{ssection:sec:maxcov}. The developed algorithm relies on solutions to shortest path problems in directed acyclic graphs, hence possesses a~low asymptotic complexity. Further, we show that a slight modification maintains an approximation ratio of $2/3$ for the tile sets whose transducer graphs are cyclic.

Section \ref{sec:results} collects results from the computational assessment of the integer programming formulations (Section \ref{ssection:sec:experiments}) and heuristics (Section \ref{ssection:sec:exp_heuristics}), and on the benchmarking of the periodic tile packing formulation against the algorithm of Lagae and Dutr\'e~\cite{lagae2007} (Section \ref{ssection:sec:tilepacking}). We close the section with two surprising observations found with integer programming for two well-known aperiodic tile sets: the Knuth~\cite{knuth1968} tile set of 92 tiles contains a~tile unusable in infinite simply-connected valid tilings, Section \ref{ssection:sec:knuth}, and the Lagae \textit{et al.}~\cite{lagae2006report} tile set of 44 corner tiles is not aperiodic, Section \ref{ssection:sec:aperiodic}. We summarize our results in Section \ref{sec:wangconc}.

\section{Notation and preliminaries}\label{sec:definitions}

Considering a~finite set of \textit{color codes} $\mathcal{C} = \{1, 2,\dots, n_\mathrm{c}\} \subset \mathbb{N}$, the \textit{(Wang) tile} $k$ is a~quadruple of the color codes $(c^\mathrm{n}_{k}, c^\mathrm{w}_{k}, c^\mathrm{s}_{k}, c^\mathrm{e}_{k})$, with $c^\mathrm{n}_{k}, c^\mathrm{w}_{k}, c^\mathrm{s}_{k}$, and $c^\mathrm{e}_{k} \in \mathcal{C}$ standing for the color codes of the north, west, south, and east edge of the tile $k$, respectively. Tiles can, therefore, be represented graphically as non-rotatable squares shown in Fig.~\ref{fig:Wang_tile}. Without loss of generality, we further consider these squares to be of the unit size.

\begin{figure}[!t]
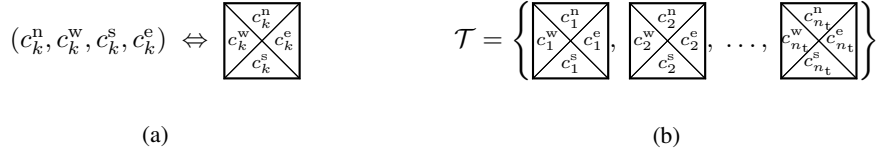

	\vspace*{-7pt}
	\begin{subfigure}{0.5\linewidth}
		\centering
		\begin{equation*}
		(c_k^\mathrm{n}, c_k^\mathrm{w}, c_k^\mathrm{s}, c_k^\mathrm{e})\; \Leftrightarrow\; \WangTile{$c_k^\mathrm{n}$}{$c_k^\mathrm{w}$}{$c_k^\mathrm{s}$}{$c_k^\mathrm{e}$}
		\end{equation*}
		\caption{}
		\label{fig:Wang_tile}
	\end{subfigure}%
	\begin{subfigure}{0.5\linewidth}
		\centering
		\begin{equation*}
		\mathcal{T} = 
		\left\{ 
		\WangTile{$c_1^\mathrm{n}$}{$c_1^\mathrm{w}$}{$c_1^\mathrm{s}$}{$c_1^\mathrm{e}$},\;
		\WangTile{$c_2^\mathrm{n}$}{$c_2^\mathrm{w}$}{$c_2^\mathrm{s}$}{$c_2^\mathrm{e}$},\;
		\dots,\;
		\WangTile{$c_{n_\mathrm{t}}^\mathrm{n}$}{$c_{n_\mathrm{t}}^\mathrm{w}$}{$c_{n_\mathrm{t}}^\mathrm{s}$}{$c_{n_\mathrm{t}}^\mathrm{e}$}
		\right\}
		\end{equation*}
		\caption{}
		\label{fig:tile_set}
	\end{subfigure}
	\vspace{-5pt}
	\caption{Graphical representation of (a) a~Wang tile $k$, and of (b) a tile set $\mathcal{T}$.}
	\vspace*{-5pt}
\end{figure}

A~\textit{tile set} $\mathcal{T}$ represents a~finite collection of $n_\mathrm{t}$ tiles, see Fig.~\ref{fig:tile_set}. When $\forall (c^\mathrm{n},c^\mathrm{w},c^\mathrm{s},c^\mathrm{e}) \in \mathcal{C}^4: (c^\mathrm{n},c^\mathrm{w},c^\mathrm{s},c^\mathrm{e}) \in \mathcal{T}$, we call the tile set \textit{complete}.

Using the notation $\tilde{\bullet} = \bullet \cap \mathbb{Z}^2$ to denote an intersection of the set $\bullet$ with the integer lattice points, \textit{tiling} $\mathfrak{T}^\mathcal{A}$ of a~bounded domain $\mathcal{A} \in \mathbb{R}^2$ is an arrangement of copies of the tiles from $\mathcal{T}$ such that the tiles are placed at $\tilde{\mathcal{A}}$, and cover the entire domain $\mathcal{A}$, cf. Fig.~\ref{fig:edge_matching}. More formally, tiling is a~mapping $\mathfrak{T}^\mathcal{A}: \tilde{\mathcal{A}} \rightarrow \mathcal{T}$ assigning a~single tile $k \in \mathcal{T}$ to every coordinate $\left(i,j\right) \in \tilde{\mathcal{A}}$. Consequently, we call tilings $\mathfrak{T}^\mathcal{A}$ \emph{simply connected} iff the domain $\mathcal{A}$ is so.
\begin{figure}[!b]
	\vspace*{-7pt}
	\centering
	\begin{tikzpicture}[scale=1]
	\draw[step=1,black,xshift=0,yshift=0,very thin,dashed] (0,0) grid (3.25,-2.25);
	\draw[step=1,black,xshift=0,yshift=0,very thin,dashed] (3.75,0) grid (4.0,-2.25);
	\draw[step=1,black,xshift=0,yshift=0,very thin,dashed] (0,-3) grid (3.25,-2.75);
	\draw[step=1,black,xshift=0,yshift=0,very thin,dashed] (3.75,-3) grid (4,-2.75);
	\draw (0,0)--(3.25,0);
	\draw[decoration={markings,mark=at position 1 with
		{\arrow[scale=1.5,>=stealth]{>}}},postaction={decorate}] (3.75,0)--(4.5,0) node[above right]{$\mathcal{W}$};
	\draw (0,0)--(0,-2.25);
	\draw[decoration={markings,mark=at position 1 with
		{\arrow[scale=1.5,>=stealth]{>}}},postaction={decorate}] (0,-2.75)--(0,-3.5) node[left]{$\mathcal{H}$};
	\node[label=$1$] at (0,0) {}; 
	\node[label=$2$] at (1,0) {}; 
	\node[label=$3$] at (2,0) {}; 
	\node[label=$4$] at (3,0) {}; 
	\node[label=$n_\mathrm{w}$] at (4,0) {}; 
	\node[label=left:$1$] at (0,0) {}; 
	\node[label=left:$2$] at (0,-1) {}; 
	\node[label=left:$3$] at (0,-2) {}; 
	\node[label=left:$n_\mathrm{h}$] at (0,-3) {};
	\WT{$c^\mathrm{n}_p$}{$c^\mathrm{w}_p$}{$c^\mathrm{s}_p$}{$c^\mathrm{e}_p$}{0.5}{-1.5}
	\WT{$c^\mathrm{n}_q$}{$c^\mathrm{w}_q$}{$c^\mathrm{s}_q$}{$c^\mathrm{e}_q$}{0.5}{-2.5}
	\WT{$c^\mathrm{n}_r$}{$c^\mathrm{w}_r$}{$c^\mathrm{s}_r$}{$c^\mathrm{e}_r$}{1.5}{-1.5}
	\end{tikzpicture}
	\caption{Color matching among tiles $p$, $q$, and $r \in \mathcal{T}$.}
	\label{fig:edge_matching}
	\vspace*{-5pt}
\end{figure}
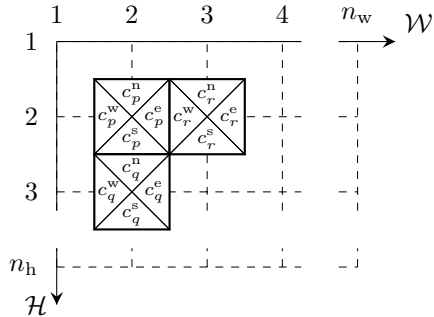

The tiling $\mathfrak{T}^\mathcal{A}$ is \textit{rectangular} if $\forall i \in \mathcal{H}, \mathcal{H} = \{1,\dots, n_\mathrm{h}\},$ and $\forall j \in \mathcal{W}, \mathcal{W} = \{1,\dots, n_\mathrm{w}\}$, it holds that $(i,j) \in \tilde{\mathcal{A}}$. Here, $\mathcal{H}$ and $\mathcal{W}$ are the sets of the height and width coordinates.

A \textit{valid tiling} (Wang tiling) of $\mathcal{A}$, denoted by $\mathfrak{T}^\mathcal{A}_\mathrm{valid}$, is a~tiling with equal color codes at the shared edges between all pairs of adjoining tiles. Therefore, the mapping $\mathfrak{T}^\mathcal{A}_\mathrm{valid}: \tilde{\mathcal{A}} \rightarrow \mathcal{T}$ satisfies, in addition to the requirements for $\mathfrak{T}^\mathcal{A}$, the additional constraints
\begin{subequations}
	\begin{equation}\label{eq:southNorthMapping}
	c^\mathrm{s}_{\mathfrak{T}^\mathcal{A}_\mathrm{valid}(i,j)} = c^\mathrm{n}_{\mathfrak{T}^\mathcal{A}_\mathrm{valid}(i+1,j)}, \quad \forall (i,j) \in \tilde{\mathcal{A}}: (i+1,j) \in \tilde{\mathcal{A}},
	\end{equation}%
	\begin{equation}\label{eq:eastWestMapping}
	c^\mathrm{e}_{\mathfrak{T}^\mathcal{A}_\mathrm{valid}(i,j)} = c^\mathrm{w}_{\mathfrak{T}^\mathcal{A}_\mathrm{v}(i,j+1)}, \quad \forall (i,j) \in \tilde{\mathcal{A}}: (i,j+1) \in \tilde{\mathcal{A}},
	\end{equation}
\end{subequations}
provided that the axes are oriented accordingly to Fig.~\ref{fig:edge_matching}. If such $\mathfrak{T}^\mathcal{A}_\mathrm{valid}$ exists, we say that the domain $\mathcal{A}$ admits a~valid $\mathcal{T}$-tiling, or that it is \textit{tileable} by $\mathcal{T}$.

Consider that $\mathcal{B}\subseteq \mathcal{A}$ and $\mathcal{B_{\max \mathrm{rect}}} \subseteq \mathcal{A}$ are simply connected, rectangular, and $\mathcal{T}$-tileable. Then, the \textit{maximum rectangular valid tiling} $\mathfrak{T}^\mathcal{A}_{\mathrm{v},\max \mathrm{rect}}$ is a~valid tiling of the domain $\mathcal{B}_{\max \mathrm{rect}}$, where $\{\mathcal{B}_{\max \mathrm{rect}} \subseteq \mathcal{A}, \forall \mathcal{B} \subseteq \mathcal{A}: \lvert\tilde{\mathcal{B}}_{\max \mathrm{rect}} \rvert \ge \lvert \tilde{\mathcal{B}} \rvert\}$. Here, the notation $\lvert \bullet \rvert$ denotes cardinality of the set $\bullet$.

The \textit{maximum cover} $\mathfrak{T}^\mathcal{A}_\mathrm{v, {\max \mathrm{cov}}}$ is a~valid tiling of $\mathcal{B_{\max \mathrm{cov}}}$, where $\mathcal{B}$ and $\mathcal{B}_{\max \mathrm{cov}}$ are arbitrary $\mathcal{T}$-tileable subdomains in $\mathcal{A}$ and $\{\mathcal{B}_{\max \mathrm{cov}} \subseteq \mathcal{A}, \forall \mathcal{B} \subseteq \mathcal{A}: \lvert\tilde{\mathcal{B}}_{\max \mathrm{cov}} \rvert \ge \lvert \tilde{\mathcal{B}} \rvert\}$.

A rectangular valid tiling is said to be \textit{periodic}, if the color codes at the opposite sides of the rectangle match. If the valid tiling is not periodic, but the considered tile set allows for at least one periodic rectangular tiling, we call it \textit{nonperiodic}. Finally, if no such periodic pattern exists and the tile set still allows for a~valid tiling of the infinite plane, it is referred to as \textit{aperiodic}. Similarly, the tile set $\mathcal{T}$ is \textit{periodic} if it permits periodic valid tilings; and \textit{aperiodic} if all feasible valid tilings are aperiodic.

\textit{Transducer graph}\cite{kari1996} $G_\mathrm{t,h}$ of the tile set $\mathcal{T}$ is a~directed (multi-)graph representation of a~Mealy machine without any initial nor terminal state. It consists of $\lvert \mathcal{C}\rvert$ states (graph vertices) and $\lvert \mathcal{T} \rvert$ transitions (directed edges) $\mathcal{E}_\mathrm{h}$, where
\begin{equation}
\mathcal{E}_\mathrm{h} \coloneqq \bigcup_{k \in \mathcal{T}}\left( c_k^\mathrm{w} \xrightarrow{c_k^\mathrm{s}\vert c_k^\mathrm{n}} c_k^\mathrm{e}\right).
\end{equation}
For the \textit{dual transducer graph} $G_\mathrm{t,v}$, composed of the dual Wang tiles~\cite{labbe2018} reflecting $\mathcal{T}$ along the major diagonal of the tiles, the edge set is defined as 
\begin{equation}
\mathcal{E}_\mathrm{v} \coloneqq \bigcup_{k \in \mathcal{T}}\left( c_k^\mathrm{n} \xrightarrow{c_k^\mathrm{e}\vert c_k^\mathrm{w}} c_k^\mathrm{s}\right).
\end{equation}
To illustrate the construction, we include a~visual example in Fig.~\ref{fig:transducergraphs}.

\begin{figure}[!t]
	\vspace*{-7pt}
	\begin{subfigure}[t]{0.333\linewidth}
		\centering
		\raisebox{6mm}{$\mathcal{T} = \left\{\WangTile{0}{1}{1}{0},
			\WangTile{0}{1}{0}{1},
			\WangTile{1}{0}{0}{1}
			\right\}$}
		\caption{}
	\end{subfigure}%
	\begin{subfigure}[t]{0.333\linewidth}
		\centering
		\begin{tikzpicture}[>=stealth',shorten >=1pt,auto,node distance=2cm, semithick,every node/.style={inner sep=2pt}, every state/.style={minimum size=0pt,inner sep=2pt}]
		\node (0) [state] {$0$};
		\node (1) [state, right=of 0] {$1$};
		\path[->] (0) edge [bend left=40, above, sloped, rotate=180] node(b) {$0\vert1$} (1);
		\path[->] (1) edge [bend left=40, above, sloped, rotate=180] node(b) {$1\vert0$} (0);
		\path[->] (1)  edge [loop above, sloped] node {$0\vert0$} (1);
		\end{tikzpicture}
		\caption{}
	\end{subfigure}%
	\begin{subfigure}[t]{0.333\linewidth}
		\centering
		\begin{tikzpicture}[>=stealth',shorten >=1pt,auto,node distance=2cm, semithick,every node/.style={inner sep=2pt}, every state/.style={minimum size=0pt,inner sep=2pt}]
		\node (2) [state] {$0$};
		\node (3) [state, right=of 0] {$1$};
		\path[->] (0) edge [bend left=40, above, sloped, rotate=180] node(b) {$0\vert1$} (1);
		\path[->] (1) edge [bend left=40, above, sloped, rotate=180] node(b) {$1\vert0$} (0);
		\path[->] (0)  edge [loop above, sloped] node {$1\vert1$} (0);
		\end{tikzpicture}
		\caption{}
	\end{subfigure}%
	\vspace{-5pt}
	\caption{(b) Transducer and (c) dual transducer graphs of the tile set (a).}
	\label{fig:transducergraphs}
	\vspace*{-7pt}
\end{figure}
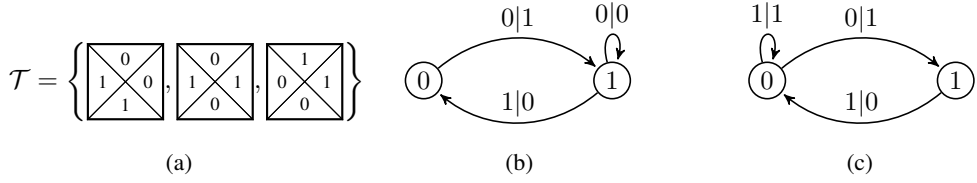

\section{Integer programming formulations}\label{sec:intprog}

In this section, we introduce four integer programming formulations for the generation of valid tilings. The first one, in Section \ref{ssection:sec:feasibility}, develops a~decision variant. In the later sections, we investigate the maximum rectangular tiling (Section \ref{ssection:sec:maxrect}), maximum cover (Section \ref{ssection:sec:maxcov}), and the maximum adjacency constraints satisfaction (Section \ref{ssection:sec:maxadjacency}). Finally, Section \ref{ssection:sec:extensions} proposes several extensions to facilitate finer control over the resulting tilings.

\subsection{Rectangular valid tiling}\sslabel{sec:feasibility}

Let us now consider the fundamental problem of finding $\mathfrak{T}^\mathcal{A}_\mathrm{v}$ or proving it does not exist. From now on, we restrict $\mathcal{A}$ to be rectangular to simplify notation. However, the presented approach also extends to the general case.

To achieve this, we introduce $\forall (i,j,k) \in \mathcal{H}\times \mathcal{W}\times\mathcal{T}$ a~binary decision variable $x_{i,j,k} \in \{0,1\}$ denoting the placement of the tile $k$ at the $(i,j)$ coordinate such that
\begin{equation}\label{eq:integrality}
x_{i,j,k}  = \left\{
\begin{array}{c l}
1 & \text{iff the tile $k$ lies at coordinate $(i,j)$,}\\
0 & \text{otherwise.}
\end{array}\right.
\end{equation}
Consequently, mapping $\mathfrak{T}^\mathcal{A}(i,j)$ is expressed as
\begin{equation}\label{eq:mapping}
\mathfrak{T}^\mathcal{A}(i,j) = \sum_{k \in \mathcal{T}} k x_{i,j,k},
\end{equation}
together with the requirement that every $(i,j)$ coordinate is occupied by one tile,
\begin{equation}\label{eq:sumTiles}
\sum_{k \in \mathcal{T}} x_{i,j,k} = 1,\quad \forall (i,j) \in \mathcal{H} \times \mathcal{W}.
\end{equation}
Similarly, the color codes of a~tile placed at $(i,j)$ are expressed using the binary variables as
\begin{subequations}\label{eq:colors}
\begin{minipage}{0.48\linewidth}
	\centering
	\begin{equation}\label{eq:northMap}
	c_{\mathfrak{T}^\mathcal{A}(i,j)}^\mathrm{n} = \sum_{k \in \mathcal{T}} c_k^\mathrm{n} x_{i,j,k},
	\end{equation}
	\begin{equation}\label{eq:westMap}
	c_{\mathfrak{T}^\mathcal{A}(i,j)}^\mathrm{w} = \sum_{k \in \mathcal{T}} c_k^\mathrm{w} x_{i,j,k},
	\end{equation}
\end{minipage}%
\hfill\begin{minipage}{0.48\linewidth}
	\begin{equation}\label{eq:southMap}
	c_{\mathfrak{T}^\mathcal{A}(i,j)}^\mathrm{s} = \sum_{k \in \mathcal{T}} c_k^\mathrm{s} x_{i,j,k},
	\end{equation}
	\begin{equation}\label{eq:eastMap}
	c_{\mathfrak{T}^\mathcal{A}(i,j)}^\mathrm{e} = \sum_{k \in \mathcal{T}} c_k^\mathrm{e} x_{i,j,k}.
	\end{equation}
	\end{minipage}
\end{subequations}
Inserting \eqref{eq:colors} into \eqref{eq:southNorthMapping} and \eqref{eq:eastWestMapping} leads to the horizontal and vertical adjacency constraints expressed in terms of the decision variables, as
\begin{subequations}\label{eq:colorCon0}
	\begin{equation}
	\sum_{k \in \mathcal{T}} c_k^\mathrm{s} x_{i,j,k} - \sum_{k \in \mathcal{T}} c_k^\mathrm{n} x_{i+1,j,k} = 0, \quad \forall (i,j) \in \mathcal{H}\setminus \{n_\mathrm{h}\}\times \mathcal{W},
	\end{equation}
	\begin{equation}
	\sum_{k \in \mathcal{T}} c_k^\mathrm{e} x_{i,j,k} - \sum_{k \in \mathcal{T}} c_k^\mathrm{w} x_{i,j+1,k} = 0, \quad \forall (i,j) \in \mathcal{H}\times \mathcal{W}\setminus \{n_\mathrm{w}\}.
	\end{equation}
\end{subequations}
Combining \eqref{eq:integrality}, \eqref{eq:mapping}, \eqref{eq:sumTiles}, and \eqref{eq:colorCon0} then provides us with a~complete binary linear programming representation of valid tiling $\mathfrak{T}^\mathcal{A}_\mathrm{valid}$.

For computational reasons, it proved to be advantageous to organize the constraints according to the color codes:
\begin{subequations}\label{eq:colorCon}
	\begin{align}
	\sum_{k \in \mathcal{T}} x_{i,j,k}[c_k^\mathrm{s}=\ell] - \sum_{k \in \mathcal{T}} x_{i+1, j, k}[c_k^\mathrm{n}=\ell] = 0,\quad \forall (i,j,\ell) \in \mathcal{H}\setminus \{n_\mathrm{h}\} \times \mathcal{W}\times \mathcal{C},\label{eq:colorCona}\\
	\sum_{k \in \mathcal{T}} x_{i,j,k}[c_k^\mathrm{e}=\ell] - \sum_{k \in \mathcal{T}} x_{i, j+1, k}[c_k^\mathrm{w}=\ell] = 0,\quad \forall (i,j,\ell) \in \mathcal{H} \times \mathcal{W}\setminus \{n_\mathrm{w}\} \times \mathcal{C}\label{eq:colorConb},
	\end{align}
\end{subequations}
where, in the Iverson notation~\cite{knuth1992}, $\sum_{k \in \mathcal{T}} x_{i,j,k} \left[ c_k^\mathrm{s} = \ell \right]$ expresses that $x_{i,j,k}$ is added to the sum if and only if $c_k^\mathrm{s} = \ell$.	

The constraint \eqref{eq:colorCona} requires that the number of tiles at $(i,j)$ with the south edge colored by $\ell$ equals to the number of tiles at $(i+1,j)$ with the north edge marked by the same $\ell$, for all $\ell \in \mathcal{C}$. Because of \eqref{eq:sumTiles}, there are either no tiles with the shared edge colored by $\ell$, or there is a~single tile at each of the coordinates with its common edge labeled by $\ell$. Analogously to the vertical adjacency constraint, the horizontal constraint \eqref{eq:colorConb} also enforces equality among the number of tiles at $(i,j)$ with the east edge colored by $\ell$ and the number of tiles at $(i,j+1)$ having the west edge colored by identical $\ell$.

Finally, combining \eqref{eq:integrality}, \eqref{eq:sumTiles}, and \eqref{eq:colorCon}, while noticing that the constraints \eqref{eq:sumTiles} naturally propagate with the adjacency constraints from the domain boundaries (compare (\ref{eq:binaryFeasibility_d},\ref{eq:binaryFeasibility_e} with \eqref{eq:sumTiles}), leads to the binary programming formulation
\begin{subequations}\label{eq:binaryFeasibility}
	\begin{flalign}
	\mathrm{find}\; & \mathbf{x}\\
	\begin{split}
	\mathrm{s.t.}\, & 
	\sum_{k \in \mathcal{T}} x_{i,j,k}[c_k^\mathrm{s}=\ell] - \sum_{k \in \mathcal{T}} x_{i+1, j, k}[c_k^\mathrm{n}=\ell] = 0,\quad \forall (i,j, \ell) \in \mathcal{H}\setminus \{n_\mathrm{h}\}\times \mathcal{W}\times \mathcal{C},
	\end{split}\label{eq:binaryFeasibility_b}\\
	\begin{split}
	& \sum_{k \in \mathcal{T}} x_{i,j,k}[c_k^\mathrm{e}=\ell] - \sum_{k \in \mathcal{T}} x_{i, j+1, k}[c_k^\mathrm{w}=\ell] = 0,\quad \forall (i,j,\ell) \in \mathcal{H} \times \mathcal{W}\setminus \{n_\mathrm{w}\} \times \mathcal{C},
	\end{split}\label{eq:binaryFeasibility_c}\\
	& \sum_{k \in \mathcal{T}} x_{i,j,k} = 1, 
	\quad \forall (i,j) \in \{1, n_\mathrm{h}\} \times \mathcal{W},\label{eq:binaryFeasibility_d}\\
	& \sum_{k \in \mathcal{T}} x_{i,j,k} = 1, 
	\quad \forall (i,j) \in \mathcal{H} \times \{1, n_\mathrm{w}\},\label{eq:binaryFeasibility_e}\\
	& x_{i,j,k} \in \{0,1\}, 
	\quad \forall (i,j,k) \in \mathcal{H} \times \mathcal{W}\times \mathcal{T},
	\end{flalign}
\end{subequations}
that provides a~complete representation of the bounded tiling problem, i.e., all valid tilings solve the integer program, and conversely, all feasible solutions to \eqref{eq:binaryFeasibility} are valid tilings. Moreover, observe that the problem consists of two totally unimodular constraints if considered independently: (\ref{eq:binaryFeasibility_c},\ref{eq:binaryFeasibility_e}) representing row tilings, and (\ref{eq:binaryFeasibility_b},\ref{eq:binaryFeasibility_d}) being column tilings. When considered simultaneously, the resulting problem becomes \NP-complete \cite{lewis1978,lewis1981}.

\subsection{Maximum rectangular valid tiling}\sslabel{sec:maxrect}

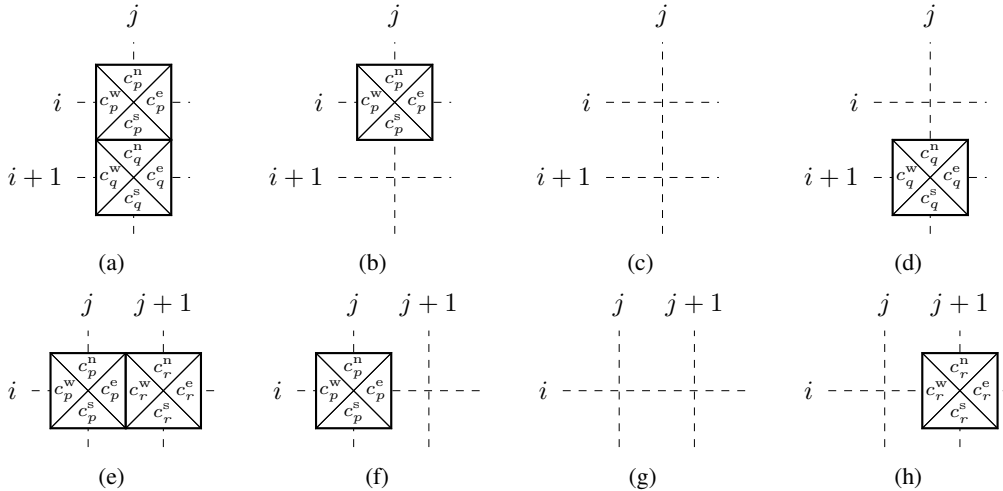
\begin{figure}[!b]
	\vspace*{-7pt}
	\begin{subfigure}{3.0cm}
		\begin{tikzpicture}[scale=1]
		\draw[step=1,black,xshift=0,yshift=0,very thin,dashed] (0.25,0.80) grid (1.75,-1.75);
		\node[label=$j$] at (1,0.75) {}; 
		\node[label=left:$i$] at (0.3,0) {}; 
		\node[label=left:$i+1$] at (0.3,-1) {}; 
		\WT{$c_p^\mathrm{n}$}{$c_p^\mathrm{w}$}{$c_p^\mathrm{s}$}{$c_p^\mathrm{e}$}{0.5}{-0.5}
		\WT{$c_q^\mathrm{n}$}{$c_q^\mathrm{w}$}{$c_q^\mathrm{s}$}{$c_q^\mathrm{e}$}{0.5}{-1.5}
		\end{tikzpicture}
		\caption{}
		\label{fig:combVerA}
	\end{subfigure}%
	\hfill
	\begin{subfigure}{3.0cm}
		\begin{tikzpicture}[scale=1]
		\draw[step=1,black,xshift=0,yshift=0,very thin,dashed] (0.25,0.80) grid (1.75,-1.75);
		\node[label=$j$] at (1,0.75) {}; 
		\node[label=left:$i$] at (0.3,0) {}; 
		\node[label=left:$i+1$] at (0.3,-1) {}; 
		\WT{$c_p^\mathrm{n}$}{$c_p^\mathrm{w}$}{$c_p^\mathrm{s}$}{$c_p^\mathrm{e}$}{0.5}{-0.5}
		\end{tikzpicture}
		\caption{}
		\label{fig:combVerB}
	\end{subfigure}
	\hfill
	\begin{subfigure}{3.0cm}
		\begin{tikzpicture}[scale=1]
		\draw[step=1,black,xshift=0,yshift=0,very thin,dashed] (0.25,0.80) grid (1.75,-1.75);
		\node[label=$j$] at (1,0.75) {}; 
		\node[label=left:$i$] at (0.3,0) {}; 
		\node[label=left:$i+1$] at (0.3,-1) {}; 
		\end{tikzpicture}
		\caption{}
		\label{fig:combVerC}
	\end{subfigure}
	\hfill
	\begin{subfigure}{3.0cm}
		\begin{tikzpicture}[scale=1]
		\draw[step=1,black,xshift=0,yshift=0,very thin,dashed] (0.25,0.80) grid (1.75,-1.75);
		\node[label=$j$] at (1,0.75) {}; 
		\node[label=left:$i$] at (0.3,0) {}; 
		\node[label=left:$i+1$] at (0.3,-1) {}; 
		\WT{$c_q^\mathrm{n}$}{$c_q^\mathrm{w}$}{$c_q^\mathrm{s}$}{$c_q^\mathrm{e}$}{0.5}{-1.5}
		\end{tikzpicture}
		\caption{}
		\label{fig:combVerD}
	\end{subfigure}\\
	
	\begin{subfigure}{3.0cm}
		\begin{tikzpicture}[scale=1]
		\draw[step=1,black,xshift=0,yshift=0,very thin,dashed] (0.25,0.80) grid (2.75,-0.75);
		\node[label=$j$] at (1,0.75) {};
		\node[label=$j+1$] at (2,0.75) {}; 
		\node[label=left:$i$] at (0.3,0) {}; 
		\WT{$c_p^\mathrm{n}$}{$c_p^\mathrm{w}$}{$c_p^\mathrm{s}$}{$c_p^\mathrm{e}$}{0.5}{-0.5}
		\WT{$c_r^\mathrm{n}$}{$c_r^\mathrm{w}$}{$c_r^\mathrm{s}$}{$c_r^\mathrm{e}$}{1.5}{-0.5}
		\end{tikzpicture}
		\caption{}
		\label{fig:combVerE}
	\end{subfigure}%
	\hfill
	\begin{subfigure}{3.0cm}
		\begin{tikzpicture}[scale=1]
		\draw[step=1,black,xshift=0,yshift=0,very thin,dashed] (0.25,0.80) grid (2.75,-0.75);
		\node[label=$j$] at (1,0.75) {};
		\node[label=$j+1$] at (2,0.75) {}; 
		\node[label=left:$i$] at (0.3,0) {}; 
		\WT{$c_p^\mathrm{n}$}{$c_p^\mathrm{w}$}{$c_p^\mathrm{s}$}{$c_p^\mathrm{e}$}{0.5}{-0.5}
		\end{tikzpicture}
		\caption{}
		\label{fig:combVerF}
	\end{subfigure}%
	\hfill
	\begin{subfigure}{3.0cm}
		\begin{tikzpicture}[scale=1]
		\draw[step=1,black,xshift=0,yshift=0,very thin,dashed] (0.25,0.80) grid (2.75,-0.75);
		\node[label=$j$] at (1,0.75) {};
		\node[label=$j+1$] at (2,0.75) {}; 
		\node[label=left:$i$] at (0.3,0) {}; 
		\end{tikzpicture}
		\caption{}
		\label{fig:combVerG}
	\end{subfigure}%
	\hfill
	\begin{subfigure}{3.0cm}
		\begin{tikzpicture}[scale=1]
		\draw[step=1,black,xshift=0,yshift=0,very thin,dashed] (0.25,0.80) grid (2.75,-0.75);
		\node[label=$j$] at (1,0.75) {};
		\node[label=$j+1$] at (2,0.75) {}; 
		\node[label=left:$i$] at (0.3,0) {}; 
		\WT{$c_r^\mathrm{n}$}{$c_r^\mathrm{w}$}{$c_r^\mathrm{s}$}{$c_r^\mathrm{e}$}{1.5}{-0.5}
		\end{tikzpicture}
		\caption{}
		\label{fig:combVerH}
	\end{subfigure}%
	\vspace{-5pt}
	\caption{Admissible tile placements (a)--(c) and (e)--(g), and forbidden placements (d) and (h) in the maximum rectangular valid tiling formulation.}
	\label{fig:combVer}
	\vspace*{-7pt}
\end{figure}

When a~solution to \eqref{eq:binaryFeasibility} cannot be found in an acceptable time period or when no such solution exists, one can resort to relaxing the requirement of a~valid tiling of $\mathcal{A}$ and search for a~valid tiling of the largest rectangular subdomain.

Without loss of generality, let us assume that the maximum rectangular valid tiling always contains an anchor tile placed at $(1,1)$, i.e.,
\begin{equation}
\sum_{k \in \mathcal{T}}  x_{1,1,k} = 1.
\label{eq:rectOneEq}
\end{equation}
On the other hand, all the other coordinates may contain a~tile or be empty, thus
\begin{equation}
\sum_{k \in \mathcal{T}} x_{i,j,k} \le 1, 
\quad \forall (i,j) \in \tilde{\mathcal{A}} \setminus (1,1).
\label{eq:rectAtMostOne}
\end{equation}

Let us now pick two vertically adjacent coordinates $(i,j)$ and $(i+1,j)$. If there is a~tile $q$ placed at $(i+1,j)$, another tile $p$ has to be placed at $(i,j)$, as, otherwise, there is no simply-connected rectangular tiling containing both the tiles at $(1,1)$ and at $(i+1,j)$. Validity of the tiling also requires identical color codes at the shared edges. On the other hand, if no tile is placed at $(i+1,j)$, a~coordinate $(i,j)$ may be either occupied or empty. The allowed and forbidden combinations are shown in Fig. \ref{fig:combVerA}--\ref{fig:combVerD}. Formally stated in terms of the decision variables, these considerations are expressed as
\begin{equation}
\sum_{k \in \mathcal{T}} x_{i,j,k}[c_k^\mathrm{s}=\ell] - \sum_{k \in \mathcal{T}} x_{i+1, j, k}[c_k^\mathrm{n}=\ell] \ge 0,\quad \forall (i,j,\ell) \in \mathcal{H}\setminus \{n_\mathrm{h}\}\times \mathcal{W}\times \mathcal{C}.
\label{eq:rectAdj1}
\end{equation}

Similar arguments hold also for the coordinates $(i,j)$ and $(i,j+1)$, resulting in the constraints
\begin{equation}
\sum_{k \in \mathcal{T}} x_{i,j,k}[c_k^\mathrm{e}=\ell] - \sum_{k \in \mathcal{T}} x_{i, j+1, k}[c_k^\mathrm{w}=\ell] \ge 0,\quad \forall (i,j,\ell) \in \mathcal{H} \times \mathcal{W}\setminus \{n_\mathrm{w}\} \times \mathcal{C}.
\label{eq:rectAdj2}
\end{equation}
The allowed and forbidden combinations for this case are shown in Fig. \ref{fig:combVerE}--\ref{fig:combVerH}.

\begin{figure}[!t]
	\vspace*{-7pt}
	\centering
	\begin{subfigure}{3.7cm}
		\begin{tikzpicture}[scale=1]
		\draw[step=1,black,xshift=0,yshift=0,very thin,dashed] (0.25,0.80) grid (2.75,-1.75);
		\node[label=$j$] at (1,0.75) {}; 
		\node[label=$j+1$] at (2,0.75) {}; 
		\node[label=left:$i$] at (0.3,0) {}; 
		\node[label=left:$i+1$] at (0.3,-1) {}; 
		\WT{$c_p^\mathrm{n}$}{$c_p^\mathrm{w}$}{$c_p^\mathrm{s}$}{$c_p^\mathrm{e}$}{0.5}{-0.5}
		\WT{$c_q^\mathrm{n}$}{$c_q^\mathrm{w}$}{$c_q^\mathrm{s}$}{$c_q^\mathrm{e}$}{0.5}{-1.5}
		\WT{$c_r^\mathrm{n}$}{$c_r^\mathrm{w}$}{$c_r^\mathrm{s}$}{$c_r^\mathrm{e}$}{1.5}{-0.5}
		\WT{$c_s^\mathrm{n}$}{$c_s^\mathrm{w}$}{$c_s^\mathrm{s}$}{$c_s^\mathrm{e}$}{1.5}{-1.5}
		\end{tikzpicture}
		\caption{}
		\label{fig:comb2dA}
	\end{subfigure}%
	\hfill
	\begin{subfigure}{3.7cm}
		\begin{tikzpicture}[scale=1]
		\draw[step=1,black,xshift=0,yshift=0,very thin,dashed] (0.25,0.80) grid (2.75,-1.75);
		\node[label=$j$] at (1,0.75) {}; 
		\node[label=$j+1$] at (2,0.75) {}; 
		\node[label=left:$i$] at (0.3,0) {}; 
		\node[label=left:$i+1$] at (0.3,-1) {}; 
		\WT{$c_p^\mathrm{n}$}{$c_p^\mathrm{w}$}{$c_p^\mathrm{s}$}{$c_p^\mathrm{e}$}{0.5}{-0.5}
		\WT{$c_q^\mathrm{n}$}{$c_q^\mathrm{w}$}{$c_q^\mathrm{s}$}{$c_q^\mathrm{e}$}{0.5}{-1.5}
		\WT{$c_r^\mathrm{n}$}{$c_r^\mathrm{w}$}{$c_r^\mathrm{s}$}{$c_r^\mathrm{e}$}{1.5}{-0.5}
		\end{tikzpicture}
		\caption{}
		\label{fig:comb2dB}
	\end{subfigure}%
	\hfill
	\begin{subfigure}{3.7cm}
		\begin{tikzpicture}[scale=1]
		\draw[step=1,black,xshift=0,yshift=0,very thin,dashed] (0.25,0.80) grid (2.75,-1.75);
		\node[label=$j$] at (1,0.75) {}; 
		\node[label=$j+1$] at (2,0.75) {}; 
		\node[label=left:$i$] at (0.3,0) {}; 
		\node[label=left:$i+1$] at (0.3,-1) {}; 
		\WT{$c_p^\mathrm{n}$}{$c_p^\mathrm{w}$}{$c_p^\mathrm{s}$}{$c_p^\mathrm{e}$}{0.5}{-0.5}
		\WT{$c_q^\mathrm{n}$}{$c_q^\mathrm{w}$}{$c_q^\mathrm{s}$}{$c_q^\mathrm{e}$}{0.5}{-1.5}
		\end{tikzpicture}
		\caption{}
		\label{fig:comb2dC}
	\end{subfigure}\\
	
	\begin{subfigure}{3.7cm}
		\begin{tikzpicture}[scale=1]
		\draw[step=1,black,xshift=0,yshift=0,very thin,dashed] (0.25,0.80) grid (2.75,-1.75);
		\node[label=$j$] at (1,0.75) {}; 
		\node[label=$j+1$] at (2,0.75) {}; 
		\node[label=left:$i$] at (0.3,0) {}; 
		\node[label=left:$i+1$] at (0.3,-1) {}; 
		\WT{$c_p^\mathrm{n}$}{$c_p^\mathrm{w}$}{$c_p^\mathrm{s}$}{$c_p^\mathrm{e}$}{0.5}{-0.5}
		\WT{$c_r^\mathrm{n}$}{$c_r^\mathrm{w}$}{$c_r^\mathrm{s}$}{$c_r^\mathrm{e}$}{1.5}{-0.5}
		\end{tikzpicture}
		\caption{}
		\label{fig:comb2dD}
	\end{subfigure}
	\hfill
	\begin{subfigure}{3.7cm}
		\begin{tikzpicture}[scale=1]
		\draw[step=1,black,xshift=0,yshift=0,very thin,dashed] (0.25,0.80) grid (2.75,-1.75);
		\node[label=$j$] at (1,0.75) {}; 
		\node[label=$j+1$] at (2,0.75) {}; 
		\node[label=left:$i$] at (0.3,0) {}; 
		\node[label=left:$i+1$] at (0.3,-1) {}; 
		\WT{$c_p^\mathrm{n}$}{$c_p^\mathrm{w}$}{$c_p^\mathrm{s}$}{$c_p^\mathrm{e}$}{0.5}{-0.5}
		\end{tikzpicture}
		\caption{}
		\label{fig:comb2dE}
	\end{subfigure}%
	\hfill
	\begin{subfigure}{3.7cm}
		\begin{tikzpicture}[scale=1]
		\draw[step=1,black,xshift=0,yshift=0,very thin,dashed] (0.25,0.80) grid (2.75,-1.75);
		\node[label=$j$] at (1,0.75) {}; 
		\node[label=$j+1$] at (2,0.75) {}; 
		\node[label=left:$i$] at (0.3,0) {}; 
		\node[label=left:$i+1$] at (0.3,-1) {}; 
		\end{tikzpicture}
		\caption{}
		\label{fig:comb2dF}
	\end{subfigure}%
	\vspace{-5pt}
	\caption{Six possible placements of tiles $p$, $q$, $r$, and $s$. The combination (b) cannot appear in any rectangular tiling.}
	\label{fig:comb2d}
	\vspace*{-7pt}
\end{figure}
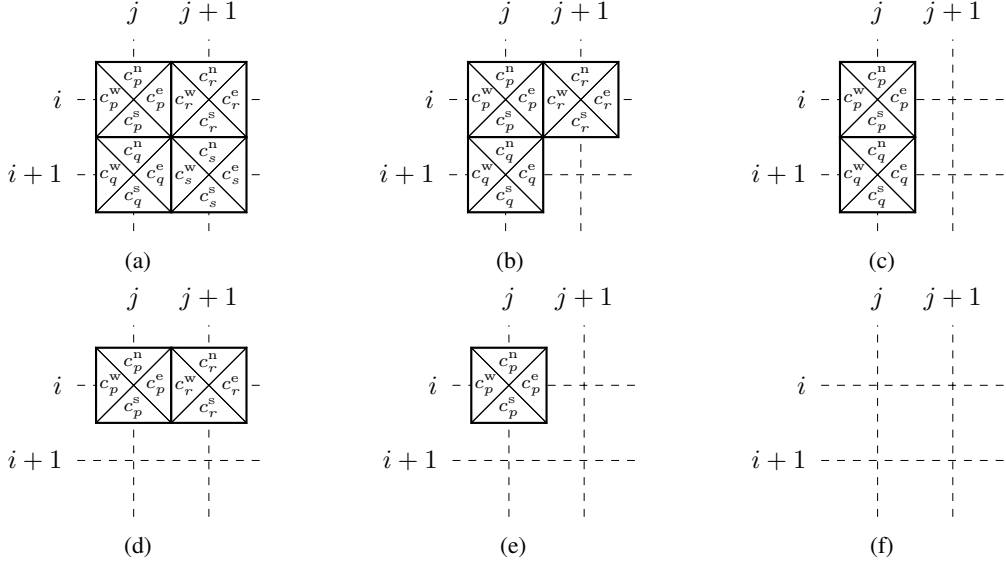

The developed constraints \eqref{eq:rectOneEq}--\eqref{eq:rectAdj2} enforce simple connectedness; however, they do not guarantee that the resultant tiling will be rectangular. For any $4$ adjacent tiles $p$, $q$, $r$, and $s$ placed at $(i,j)$, $(i+1,j)$, $(i,j+1)$, and $(i+1,j+1)$, respectively, these constraints allow for the assemblies shown in Fig. \ref{fig:comb2d}. Because the combination \ref{fig:comb2dB} cannot appear in any simply-connected rectangular tiling, we must exclude it from the feasible set,
\begin{equation}
\sum_{k \in \mathcal{T}} x_{i+1,j,k} + \sum_{k \in \mathcal{T}} x_{i,j+1,k} - \sum_{k \in \mathcal{T}} x_{i+1,j+1,k} \le 1, \quad \forall (i,j) \in \mathcal{H} \setminus\{n_\mathrm{h}\} \times \mathcal{W} \setminus\{n_\mathrm{w}\}.
\label{eq:rectRect}
\end{equation}
Finally, combining Eqs. \eqref{eq:integrality}, \eqref{eq:rectOneEq}, \eqref{eq:rectAtMostOne}, \eqref{eq:rectAdj1}, \eqref{eq:rectAdj2}, and \eqref{eq:rectRect} together with an objective function to maximize $\lvert \tilde{\mathcal{B}}_{\max \mathrm{rect}}\rvert$ provides us with the binary maximum rectangular valid tiling optimization program
\begin{subequations}\label{eq:maxRectangle}
	\begin{flalign}
	\max_{\mathbf{x}}\, & \sum_{i \in \mathcal{H}} \sum_{j \in \mathcal{W}} \sum_{k \in \mathcal{T}} x_{i,j,k}\\
	\begin{split}
	\mathrm{s.t.}\, & 
	\sum_{k \in \mathcal{T}} x_{i,j,k}[c_k^\mathrm{s}=\ell] - \sum_{k \in \mathcal{T}} x_{i+1, j, k}[c_k^\mathrm{n}=\ell] \ge 0,\quad \forall (i,j,\ell) \in \mathcal{H}\setminus \{n_\mathrm{h}\}\times \mathcal{W}\times \mathcal{C},\label{eq:maxRectangleC1}
	\end{split}\\
	\begin{split}
	& \sum_{k \in \mathcal{T}} x_{i,j,k}[c_k^\mathrm{e}=\ell] - \sum_{k \in \mathcal{T}} x_{i, j+1, k}[c_k^\mathrm{w}=\ell] \ge 0,\quad \forall (i,j,\ell) \in \mathcal{H}\times \mathcal{W}\setminus \{n_\mathrm{w}\}\times \mathcal{C},\label{eq:maxRectangleC2}
	\end{split}\\
	\begin{split}
	& \sum_{k \in \mathcal{T}} x_{i+1,j,k} + \sum_{k \in \mathcal{T}} x_{i,j+1,k} - \sum_{k \in \mathcal{T}} x_{i+1,j+1,k} \le 1, \; \forall (i,j) \in \mathcal{H} \setminus\{n_\mathrm{h}\}\times \mathcal{W} \setminus\{n_\mathrm{w}\},\label{eq:maxRectangleRect}
	\end{split}\\
	& \sum_{k \in \mathcal{T}}  x_{1,1,k} = 1,\\
	& \sum_{k \in \mathcal{T}} x_{i,j,k} \le 1, 
	\quad \forall (i,j) \in \tilde{\mathcal{A}}\setminus(1,1),\label{eq:maxRectangle1}\\
	& x_{i,j,k} \in \{0,1\}, 
	\quad \forall (i,j,k) \in \mathcal{H}\times\mathcal{W}\times \mathcal{T}.
	\end{flalign}
\end{subequations}
In contrast to \eqref{eq:binaryFeasibility}, a~feasible solution to the optimization program \eqref{eq:maxRectangle} can be found in a~polynomial time, e.g., by tiling the first row or column of the $1$D bounded tiling problem. However, finding an optimal solution to \eqref{eq:maxRectangle} is \NP-hard, because the optimization problem \eqref{eq:maxRectangle} is reducible to the decision version \eqref{eq:binaryFeasibility} by fixing the value of the objective function to $\lvert \tilde{\mathcal{A}} \rvert$, which enforces equalities in \eqref{eq:maxRectangleC1}, \eqref{eq:maxRectangleC2}, and \eqref{eq:maxRectangle1}, making the constraint \eqref{eq:maxRectangleRect} redundant as a consequence.

\subsection{Maximum cover}\sslabel{sec:maxcov}

Another option for avoiding the infeasibility of \eqref{eq:binaryFeasibility} rests in neglecting the requirement of (simple) connectedness, hence allowing for a~placement of empty tiles (voids). In this section, we therefore search the maximum cover of $\mathcal{A}$, or equivalently a~valid tiling of the (possibly disconnected) domain $\mathcal{B_{\max \mathrm{cov}}} \subseteq \mathcal{A}$. For the maximum cover formulation, we assume that any two adjacent tiles satisfy the edge-matching constraints of valid tilings, but these are also satisfied by any of the tile-void, void-tile, or void-void combination, where $\sum_{k \in \mathcal{T}} x_{i,j,k} = 0$ for a~void located at $(i,j) \in \tilde{\mathcal{A}}$.

Thus, each coordinate $(i,j)$ is occupied either by a~tile or a~void, implying that
\begin{equation}
\sum_{k \in \mathcal{T}} x_{i,j,k} \le 1, 
\quad \forall (i,j) \in \mathcal{H}\times \mathcal{W},
\label{eq:coveringAtMostOne}
\end{equation}
and the vertical and horizontal edge matching conditions become 
\begin{subequations}
	\begin{equation}
	\sum_{k \in \mathcal{T}} x_{i,j,k}[c_k^\mathrm{e}=\ell] + \sum_{k \in \mathcal{T}} x_{i, j+1, k}[c_k^\mathrm{w} \neq \ell] \le 1,\quad \forall (i,j,\ell) \in \mathcal{H}\times \mathcal{W} \setminus\{n_{\mathrm{w}}\}\times \mathcal{C},
	\label{eq:coveringAdj1}
	\end{equation}%
	\begin{equation}
	\sum_{k \in \mathcal{T}} x_{i,j,k}[c_k^\mathrm{s}=\ell] +  \sum_{k \in \mathcal{T}} x_{i+1, j, k}[c_k^\mathrm{n}\neq \ell] \le 1,\quad \forall (i,j,\ell) \in \mathcal{H} \setminus\{n_{\mathrm{h}}\}\times \mathcal{W}\times \mathcal{C}.
	\label{eq:coveringAdj2}
	\end{equation}
\end{subequations}

Finally, the combination of Eqs. \eqref{eq:coveringAtMostOne}, \eqref{eq:coveringAdj1}, \eqref{eq:coveringAdj2} with the objective function to maximize $\lvert \tilde{\mathcal{B}}_{\max \mathrm{cov}} \rvert$ leads to the binary optimization problem
\begin{subequations}
	\begin{flalign}
	\max_{\mathbf{x}}\, & \sum_{i \in \mathcal{H}} \sum_{j \in \mathcal{W}} \sum_{k \in \mathcal{T}} x_{i,j,k}\label{eq:coveringObj}\\
	\begin{split}
	\mathrm{s.t.}\, & 
	\sum_{k \in \mathcal{T}} x_{i,j,k}[c_k^\mathrm{e}=\ell] + \sum_{k \in \mathcal{T}} x_{i, j+1, k}[c_k^\mathrm{w} \neq \ell] \le 1,\quad \forall (i,j,\ell) \in \mathcal{H} \times \mathcal{W} \setminus\{n_{\mathrm{w}}\} \times \mathcal{C},
	\end{split}\label{eq:coveringCon1}\\
	\begin{split}
	& \sum_{k \in \mathcal{T}} x_{i,j,k}[c_k^\mathrm{s}=\ell] +  \sum_{k \in \mathcal{T}} x_{i+1, j, k}[c_k^\mathrm{n}\neq \ell] \le 1,\quad \forall (i,j,\ell) \in \mathcal{H} \setminus\{n_{\mathrm{h}}\} \times \mathcal{W} \times \mathcal{C},
	\end{split}\label{eq:coveringCon2}\\
	& \sum_{k \in \mathcal{T}} x_{i,j,k} \le 1, 
	\quad \forall (i,j) \in \mathcal{H} \times \mathcal{W},\\
	& x_{i,j,k} \in \{0,1\}, 
	\quad \forall (i,j,k) \in \mathcal{H} \times \mathcal{W} \times \mathcal{T}.
	\end{flalign}
	\label{eq:covering}
\end{subequations}

The program \eqref{eq:covering} is trivially \NP-hard: Requiring the objective function \eqref{eq:coveringObj} to be at least $\lvert \tilde{\mathcal{A}} \rvert$ implies that 
\begin{equation}
\sum_{k \in \mathcal{T}} x_{i,j,k} = 1, \quad \forall (i,j) \in \mathcal{H} \times \mathcal{W},
\end{equation} 
i.e., all positions are occupied by a~Wang tile. Moreover, \eqref{eq:coveringCon1} and \eqref{eq:coveringCon2} require all adjacent tiles to share the color codes at their common edges. Consequently, the resulting tiling is void-free and valid, and solves the \NP-complete bounded tiling problem.

\subsection{Maximum adjacency constraints satisfaction}\sslabel{sec:maxadjacency}

Because the decision problem \eqref{eq:binaryFeasibility} also constitutes a~specific instance of the constraint satisfaction problem (CSP), another optimization variant comes from the formulation of the max-CSP problem, maximizing the number of satisfied clauses---color matches in our case.

Therefore, for each vertical and horizontal edge we introduce a~new variable	 $h_{i,j}^\mathrm{v} \in \mathbb{R}_{\ge 0}$, where $(i,j) \in \mathcal{H} \times \mathcal{W} \setminus n_\mathrm{w}$, and $h_{i,j}^\mathrm{h} \in \mathbb{R}_{\ge 0}$, with $(i,j) \in \mathcal{H} \setminus n_\mathrm{h} \times \mathcal{W}$, respectively. The adjacency constraints \eqref{eq:colorCon} are then relaxed by considering
\begin{subequations}\label{eq:maxCSPcolors}
	\begin{align}
	\left| \sum_{k \in \mathcal{T}} x_{i,j,k}[c_k^\mathrm{s}=\ell] - \sum_{k \in \mathcal{T}} x_{i+1, j, k}[c_k^\mathrm{n}=\ell] \right| \le h_{i,j}^\mathrm{h},\quad \forall (i,j, \ell) \in \mathcal{H}\setminus \{n_\mathrm{h}\} \times \mathcal{W} \times \mathcal{C},\label{eq:maxCSPcolorA}\\
	\left| \sum_{k \in \mathcal{T}} x_{i,j,k}[c_k^\mathrm{e}=\ell] - \sum_{k \in \mathcal{T}} x_{i, j+1, k}[c_k^\mathrm{w}=\ell] \right| \le h_{i,j}^\mathrm{v},\quad \forall (i,j,\ell) \in \mathcal{H} \times \mathcal{W}\setminus \{n_\mathrm{w}\} \times \mathcal{C} \label{eq:maxCSPcolorB}
	\end{align}
\end{subequations}
instead. Indeed, if $h_{i,j}^\mathrm{h} = 0$, the edge-matching requirement of the neighboring tiles at $(i,j)$ and $(i+1,j)$ is satisfied; and it is violated otherwise. Similarly, $h_{i,j}^\mathrm{v} = 0$ guarantees color matches among the tiles at $(i,j)$ and $(i,j+1)$.

Finally, rewriting absolute values in \eqref{eq:maxCSPcolors} by two linear inequalities while supplying an objective function to maximize the number of color matches yields the binary optimization problem
\begin{subequations}\label{eq:maxCSP}
	\begin{flalign}
	\max_{\mathbf{x}}\, & \sum_{i\in \mathcal{H}} \sum_{j \in \mathcal{W} \setminus n_\mathrm{w}}  \left(1 - h_{i,j}^\mathrm{v}\right) + \sum_{i\in \mathcal{H} \setminus n_\mathrm{h}} \sum_{j \in \mathcal{W}} \left(1 - h_{i,j}^\mathrm{h}\right)\\
	\begin{split}
	\mathrm{s.t.}\, & 
	\sum_{k \in \mathcal{T}} x_{i,j,k}[c_k^\mathrm{s}=\ell] - \sum_{k \in \mathcal{T}} x_{i+1, j, k}[c_k^\mathrm{n}=\ell] \le h_{i,j}^\mathrm{h},\;\; \forall (i,j,\ell) \in \mathcal{H}\setminus \{n_\mathrm{h}\} \times \mathcal{W} \times \mathcal{C},
	\end{split}\\
	\begin{split}
	& 
	\sum_{k \in \mathcal{T}} x_{i+1, j, k}[c_k^\mathrm{n}=\ell] -\sum_{k \in \mathcal{T}} x_{i,j,k}[c_k^\mathrm{s}=\ell] \le h_{i,j}^\mathrm{h},\;\; \forall (i,j,\ell) \in \mathcal{H}\setminus \{n_\mathrm{h}\} \times \mathcal{W} \times \mathcal{C},
	\end{split}\\
	\begin{split}
	& \sum_{k \in \mathcal{T}} x_{i,j,k}[c_k^\mathrm{e}=\ell] - \sum_{k \in \mathcal{T}} x_{i, j+1, k}[c_k^\mathrm{w}=\ell] \le h_{i,j}^\mathrm{v},\;\; \forall (i,j,\ell) \in \mathcal{H} \times \mathcal{W}\setminus \{n_\mathrm{w}\} \times \mathcal{C},
	\end{split}\\
	\begin{split}
	& \sum_{k \in \mathcal{T}} x_{i, j+1, k}[c_k^\mathrm{w}=\ell] - \sum_{k \in \mathcal{T}} x_{i,j,k}[c_k^\mathrm{e}=\ell] \le h_{i,j}^\mathrm{v},\;\; \forall (i,j,\ell) \in \mathcal{H} \times \mathcal{W}\setminus \{n_\mathrm{w}\} \times \mathcal{C},
	\end{split}\\
	& \sum_{k \in \mathcal{T}} x_{i,j,k} = 1, 
	\;\; \forall (i,j) \in \mathcal{H} \times \mathcal{W},\\
	& x_{i,j,k} \in \{0,1\}, 
	\;\; \forall (i,j,k) \in \mathcal{H} \times \mathcal{W} \times \mathcal{T},
	\end{flalign}
\end{subequations}
that is \NP-hard due to the reduction to \eqref{eq:binaryFeasibility} after setting all $h_{i,j}^\mathrm{v}$ and $h_{i,j}^\mathrm{h}$ to zeros. A~feasible solution can be found in a~polynomial time by finding valid row\negthinspace/\negthinspace column tilings for each row\negthinspace/\negthinspace column, so that either term $\sum_{i\in \mathcal{H}} \sum_{j \in \mathcal{W} \setminus n_\mathrm{w}} h_{i,j}^\mathrm{v}$ or $\sum_{i\in \mathcal{H} \setminus n_\mathrm{h}} \sum_{j \in \mathcal{W}} h_{i,j}^\mathrm{h}$ equals zero.

\subsection{Extensions}\sslabel{sec:extensions}

Up to now, we have focused solely on the (re)formulations of the bounded tiling problem, searching for \textit{arbitrary} valid tilings. However, some potential applications may require finer control over the resulting tilings. Thus, in this section, we state some simple extensions to enforce tile- and color-based boundary conditions to solve the tile packing problem~\cite{lagae2007} and to enforce (variable-sized) periodic boundary conditions.

\subsubsection{Tile-based boundary conditions}\ssslabel{sec:tile_based_boundary}

At first, we consider boundary conditions in the form of prescribed tiles. As the simplest one, we enforce the placement of a~tile $k$ at $(i,j)$:
\begin{equation}
x_{i,j,k} = 1, \quad (i,j,k) \in \mathcal{H} \times \mathcal{W} \times \mathcal{T}.
\end{equation}
Similarly, we may prevent tile $k$ from being placed there:
\begin{equation}
x_{i,j,k} = 0, \quad (i,j,k) \in \mathcal{H} \times \mathcal{W} \times \mathcal{T}.
\end{equation}
Placement of an identical tile at the coordinates $(i,j) \in \tilde{\mathcal{A}}$ and $(p,q) \in \tilde{\mathcal{A}}$ requires
\begin{equation}
x_{i,j,k} -  x_{p,q,k} = 0, \quad \{i,p\} \in \mathcal{H}, \{j,q\} \in \mathcal{W}, \forall k \in \mathcal{T}.
\end{equation}
Conversely, different tiles at these coordinates are secured with
\begin{equation}
x_{i,j,k} + x_{p,q,k} \le 1,\quad\{i,p\} \in \mathcal{H}, \{j,q\} \in \mathcal{W}, \forall k \in \mathcal{T}.
\end{equation}

\subsubsection{Color-based boundary conditions}

In addition to the tile-based constraints, we may also enforce specific color codes for individual edges. To do this, the color of the north edge at $(i,j) \in \tilde{\mathcal{A}}$ is set to $\ell$ by
\begin{equation}
\sum_{k \in T} x_{i,j,k} [c^\mathrm{n}_k = \ell] = 1, \quad (i,j,\ell) \in \mathcal{H} \times \mathcal{W} \times \mathcal{C}.
\end{equation}
On the contrary, we may prevent this color by requiring
\begin{equation}
\sum_{k \in T} x_{i,j,k} [c^\mathrm{n}_k = \ell] = 0, \quad (i,j,\ell) \in \mathcal{H} \times \mathcal{W} \times \mathcal{C}.
\end{equation}
Further, the same color codes at the north edge of $(i,j) \in \tilde{\mathcal{A}}$ and at the west edge of $(p,q) \in\tilde{\mathcal{A}}$ are established with
\begin{equation}
\sum_{k \in T} x_{i,j,k} [c^\mathrm{n}_k = \ell] - \sum_{k \in T} x_{p,q,k} [c^\mathrm{w}_k = \ell] = 0, \quad \{i,p\} \in \mathcal{H}, \{j,q\} \in \mathcal{W}, \forall \ell \in \mathcal{C},
\end{equation}
and a~different color with
\begin{equation}
\sum_{k \in T} x_{i,j,k} [c^\mathrm{n}_k = \ell] + \sum_{k \in T} x_{p,q,k} [c^\mathrm{w}_k = \ell] \le 1, \quad \{i,p\} \in \mathcal{H}, \{j,q\} \in \mathcal{W}, \forall \ell \in \mathcal{C}.
\end{equation}

\subsubsection{Periodic tiling}

In the domino problem, Wang~\cite{wang1963} investigated the existence of tile sets admitting infinite aperiodic tilings. Here, we consider a~similar setting for the finite domain $\mathcal{A}$: examining periodicity through periodic color-based boundary conditions.

We begin with requiring equal coloring at the fixed opposite domain boundaries,
\begin{subequations}\label{eq:periodicfixed}
	\begin{align}
	\sum_{k \in T} x_{1,j,k} [n_k = \ell] - \sum_{k \in T} x_{n_{\mathrm{t},h},j,k} [s_k = \ell] = 0, \quad \forall (j,\ell) \in \mathcal{W} \times \mathcal{C},\\
	\sum_{k \in T} x_{i,1,k} [w_k = \ell] - \sum_{k \in T} x_{i,n_{\mathrm{t},w},k} [e_k = \ell] = 0, \quad \forall (i,\ell) \in \mathcal{H} \times \mathcal{C}.
	\end{align}
\end{subequations}
When adding \eqref{eq:periodicfixed} to the decision problem \eqref{eq:binaryFeasibility}, we thus ask for an existence of a~fixed-sized periodic Wang tiling.

In a~natural generalization, we ask for an existence of finite-sized periodic Wang tilings, thus relying on the maximum rectangular valid tiling formulation \eqref{eq:maxRectangle}. Naturally, the domain size is not known in this case. Therefore, we must consider $\forall (i,j,\ell) \in \mathcal{H} \times \mathcal{W}\times \mathcal{C}$ constraints of the form
\begin{subequations}\label{eq:periodicRectangle}
	\begin{flalign}
	\sum_{k \in \mathcal{T}} x_{i,j,k}[e_k\neq \ell] + \sum_{k \in \mathcal{T}} x_{i, 1, k}[w_k=\ell] -  \sum_{k \in \mathcal{T}} x_{i,j+1,k} [j<n_{\mathrm{t},w}] \le 1,\label{eq:periodicRectangle1}\\
	\sum_{k \in \mathcal{T}} x_{i,j,k}[s_k \neq \ell] + \sum_{k \in \mathcal{T}} x_{1, j, k}[n_k=\ell] - \sum_{k \in \mathcal{T}} x_{i+1,j,k}[i<n_{\mathrm{t},h}] \le 1.\label{eq:periodicRectangle2}
	\end{flalign}
\end{subequations}
Here, \eqref{eq:periodicRectangle1} prevents a~color mismatch of the north edge of $(1,j) \in \tilde{\mathcal{A}}$ and the south edge of $(i,j) \in \mathcal{A}$ iff there is no tile placed at $(i,j+1) \in \tilde{\mathcal{A}}$. Similarly, in the case of \eqref{eq:periodicRectangle2}, we prevent a~color mismatch of the west edge at $(i,1) \in \tilde{\mathcal{A}}$ and the east edge at $(i,j) \in \tilde{\mathcal{A}}$ iff the position $(i+1,j) \in \tilde{\mathcal{A}}$ is empty.

Finally, when adding the constraints \eqref{eq:periodicRectangle} to \eqref{eq:maxRectangle}, we usually search for the smallest periodic pattern rather than the largest,
\begin{align}\label{eq:smallestTiling}
\min_{\mathbf{x}}\, & \sum_{i \in \mathcal{H}} \sum_{j \in \mathcal{W}} \sum_{k \in \mathcal{T}} x_{i,j,k}.
\end{align}

\subsubsection{Tile packing problem}

Our last extension constitutes the setting of the tile-packing problem \cite{lagae2007}: we require each tile to be placed exactly once yet form a~fixed-sized valid tiling,
\begin{equation}\label{eq:packing}
\sum_{i \in \mathcal{H}} \sum_{j \in \mathcal{W}} x_{i,j,k} = 1, \quad \forall k \in \mathcal{T}.
\end{equation}
Note here that this extension requires that $\lvert \mathcal{T} \rvert = \lvert \tilde{\mathcal{A}} \rvert$ as, otherwise, no solution exists.

\section{Heuristic algorithm for the maximum cover tiling problem}\label{sec:heuristics}

In the previous sections, we have introduced several integer programming formulations for the bounded Wang tiling problem and their extensions. Because of the asymptotic complexity of the integer programming formulations, we further develop a~simple heuristic algorithm for one of the optimization variants, the maximum cover.

\subsection{Maximum row cover tilings}\sslabel{sec:maximumrowcover}

Let us start with revising the decision program~\eqref{eq:binaryFeasibility}. In this formulation, neglecting any pair of the constraints (\ref{eq:binaryFeasibility_b}, \ref{eq:binaryFeasibility_d}) or (\ref{eq:binaryFeasibility_c},\ref{eq:binaryFeasibility_e}) provides a~totally unimodular constraint matrix, recall Section \ref{ssection:sec:feasibility}. Consequently, such simplified problems are deterministically solvable using the simplex method. Moreover, this setting agrees with the maximum flow problem \cite{korte2006}, as \eqref{eq:binaryFeasibility_d} and \eqref{eq:binaryFeasibility_e} correspond to the flow balances in the source and sink, whereas \eqref{eq:binaryFeasibility_b} and \eqref{eq:binaryFeasibility_c} correspond to the Kirchhoff law equations. Further complexity reduction is possible by recognizing the (shortest) path problem structure, since the source and sink capacities are equal to one, allowing only a~single source-to-sink path with positive flow to emerge. Omitting any of these constraint pairs produces valid tilings of (finite) stripes, i.e., of rows or columns. However, the edges shared by the neighboring stripes may not comply with the edge matching rules. Starting with this observation, we first focus on an efficient approach to generate valid tilings of the rows.

\begin{figure}[!t]
	\vspace*{-7pt}
	\centering
	\resizebox{0.65\linewidth}{!}{
	\begin{tikzpicture}[>=stealth',shorten >=1pt,auto,node distance=2cm, semithick,every node/.style={inner sep=2pt}, every state/.style={minimum size=20pt,inner sep=2pt},on grid]
	% layer 1
	\node (s) [state] at (0,0) {$s$};
	% layer2
	\node (0) [state] at (2,2.25) {$1$};
	\node (1) [state] at (2,0.75) {$2$};
	\node (d) [] at (2,-0.75) {$\vdots$};
	\node (C) [state] at (2,-2.25) {$\lvert\mathcal{C}\rvert$};
	% layer 3
	\node (0a) [state, right of=0]{$1$};
	\node (1a) [state, right of=1]{$2$};
	\node (da) [right of=d] {$\vdots$};
	\node (Ca) [state, right of=C]{$\lvert\mathcal{C}\rvert$};
	% layer 4
	\node (0b) [right of=0a]{$\dots$};
	\node (1b) [right of=1a]{$\dots$};
	\node (db) [right of=da] {$\ddots$};
	\node (Cb) [right of=Ca]{$\dots$};
	% layer 5
	\node (0c) [state, right of=0b]{$1$};
	\node (1c) [state, right of=1b]{$2$};
	\node (dc) [right of=db] {$\vdots$};
	\node (Cc) [state, right of=Cb]{$\lvert\mathcal{C}\rvert$};
	% terminal
	\node (t) [state] at (10,0){$t$};
	% edges
	\node (t1) [rotate=90] at (3,0) {edges $\mathcal{E}_\mathrm{h}$};
	\node (t2) [rotate=90] at (5,0) {edges $\mathcal{E}_\mathrm{h}$};
	\node (t3) [rotate=90] at (7,0) {edges $\mathcal{E}_\mathrm{h}$};
	\path[->] (s) edge (0);
	\path[->] (s) edge (1);
	\path[->] (s) edge (C);
	\path[->] (0c) edge (t);
	\path[->] (1c) edge (t);
	\path[->] (Cc) edge (t);
	% labels
	\node (l1) [] at (3,3) {Tile $1$};
	\node (l2) [] at (5,3) {Tile $2$};
	\node (l3) [] at (7,3) {Tile $\lvert \mathcal{W} \rvert$};
	\end{tikzpicture}}
	\caption{Transducer-based directed acyclic graph for generation of valid row tilings.}
	\label{fig:transducerG1}
	\vspace*{-7pt}
\end{figure}
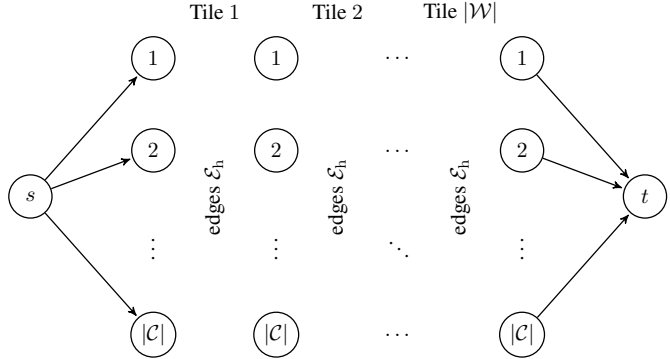

As follows from Section \ref{sec:definitions}, any valid tiling of a~row can be visualized as a~$\lvert \mathcal{W} \rvert$-long path in the transducer graph $G_\mathrm{t,h}$, recall Section~\ref{sec:definitions}. To simplify subsequent developments, we represent the row-tiling problem by a~transducer-based directed acyclic graph (DAG) composed of $\lvert \mathcal{W}\rvert +3$ vertex layers. While both the first and the last layer contain only a~single vertex (the source $s$ and terminal $t$), the intermediate layers include $\lvert \mathcal{C} \rvert$ vertices to represent the vertical (east and west) color codes of the tiles, i.e., the states in the transducer graph. The source vertex is connected to all vertices in the second layer, facilitating an arbitrary coloring of the west edge of the first tile, and, similarly, all the vertices in the penultimate layer are linked to the terminal to allow for all colors in the last east edge. The intermediate layers are bridged with the transducer edges $\mathcal{E}_\mathrm{h}$; see Fig.~\ref{fig:transducerG1}. Consequently, any $s\rightarrow t$ path in the yet-established directed graph forms a~valid tiling of the row, and conversely, any valid tiling builds a~$s\rightarrow t$ path.

However, because such paths do not exist for tile sets that forbid a~valid tiling of the row, we also need to incorporate voids. Clearly, we can add ``void'' tiles as edges that would interconnect the layers, i.e., any two consecutive layers would form a~complete bipartite graph. However, such an approach requires adding at most $\lvert \mathcal{W} \rvert \lvert \mathcal{C} \rvert^2$ edges to the graph. Therefore, we add supplementary intermediate layers with a~single vertex only, symbolizing the ``void'' tile type, and connect it to all vertices in the preceding and subsequent layer, see the dashed vertices and edges in Fig.~\ref{fig:transducerG2}. Consequently, we generate at most $2 \lvert \mathcal{W} \rvert \lvert \mathcal{C} \rvert$ new edges altogether.

\begin{figure}[!b]
	\vspace*{-7pt}
	\centering
	\resizebox{\linewidth}{!}{
		\begin{tikzpicture}[>=stealth',shorten >=1pt,auto,node distance=2cm, semithick,every node/.style={inner sep=2pt}, every state/.style={minimum size=20pt,inner sep=2pt},on grid,node distance=4cm]
		% layer 1
		\node (s) [state] at (0,0) {$s$};
		% layer2
		\node (0) [state] at (2,2.25) {$0$};
		\node (1) [state] at (2,0.75) {$1$};
		\node (d) [] at (2,-0.75) {$\vdots$};
		\node (C) [state] at (2,-2.25) {$\lvert\mathcal{C}\rvert$};
		% layer 3
		\node (0a) [state, right of=0]{$0$};
		\node (1a) [state, right of=1]{$1$};
		\node (da) [right of=d] {$\vdots$};
		\node (Ca) [state, right of=C]{$\lvert\mathcal{C}\rvert$};
		% layer 4
		\node (0b) [right of=0a]{$\dots$};
		\node (1b) [right of=1a]{$\dots$};
		\node (db) [right of=da] {$\ddots$};
		\node (Cb) [right of=Ca]{$\dots$};
		% layer 5
		\node (0c) [state, right of=0b]{$0$};
		\node (1c) [state, right of=1b]{$1$};
		\node (dc) [right of=db] {$\vdots$};
		\node (Cc) [state, right of=Cb]{$\lvert\mathcal{C}\rvert$};
		% terminal
		\node (t) [state] at (16,0){$t$};
		% void nodes
		\node (v0) [state, dashed] at (4,0){void};
		\node (v1) [state, dashed] at (8,0){void};
		\node (v2) [state, dashed] at (12,0){void};
		% edges
		\node (t1) [align=center] at (4,-2) {+ edges $\mathcal{E}_\mathrm{h}$};
		\node (t2) [align=center] at (8,-2) {+ edges $\mathcal{E}_\mathrm{h}$};
		\node (t3) [align=center] at (12,-2) {+ edges $\mathcal{E}_\mathrm{h}$};
		\path[->] (s) edge [above,sloped] node {0.0} (0);
		\path[->] (s) edge [above,sloped] node {0.0} (1);
		\path[->] (s) edge [above,sloped] node {0.0} (C);
		\path[->] (0c) edge [above,sloped] node {0.0} (t);
		\path[->] (1c) edge [above,sloped] node {0.0} (t);
		\path[->] (Cc) edge [above,sloped] node {0.0} (t);
		% void edges
		\path[->, dashed] (0) edge [above,sloped] node {1.0} (v0);
		\path[->, dashed] (1) edge [above,sloped] node {1.0} (v0);
		\path[->, dashed] (C) edge [above,sloped] node {1.0} (v0);
		\path[->, dashed] (v0) edge [above,sloped] node {0.0} (0a);
		\path[->, dashed] (v0) edge [above,sloped] node {0.0} (1a);
		\path[->, dashed] (v0) edge [above,sloped] node {0.0} (Ca);
		\path[->, dashed] (0a) edge [above,sloped] node {1.0} (v1);
		\path[->, dashed] (1a) edge [above,sloped] node {1.0} (v1);
		\path[->, dashed] (Ca) edge [above,sloped] node {1.0} (v1);
		\path[dashed] (v1) edge [above,sloped] node {0.0} ($(v1)!0.6!(0b)$);
		\path[dashed] (v1) edge [above,sloped] node {0.0} ($(v1)!0.6!(1b)$);
		\path[dashed] (v1) edge [above,sloped] node {0.0} ($(v1)!0.6!(Cb)$);
		\path[->, dashed] ($(0b)!0.4!(v2)$) edge [above,sloped] node {1.0} (v2);
		\path[->, dashed] ($(1b)!0.4!(v2)$) edge [above,sloped] node {1.0} (v2);
		\path[->, dashed] ($(Cb)!0.4!(v2)$) edge [above,sloped] node {1.0} (v2);
		\path[->, dashed] (v2) edge [above,sloped] node {0.0} (0c);
		\path[->, dashed] (v2) edge [above,sloped] node {0.0} (1c);
		\path[->, dashed] (v2) edge [above,sloped] node {0.0} (Cc);
		% labels
		\node (l1) [] at (4,3) {Tile $1$};
		\node (l2) [] at (8,3) {Tile $2$};
		\node (l3) [] at (12,3) {Tile $\lvert \mathcal{W} \rvert$};
		\end{tikzpicture}}
	\caption{Transducer-based directed acyclic graph for computing the maximum row cover.}
	\label{fig:transducerG2}
	\vspace*{-7pt}
\end{figure}
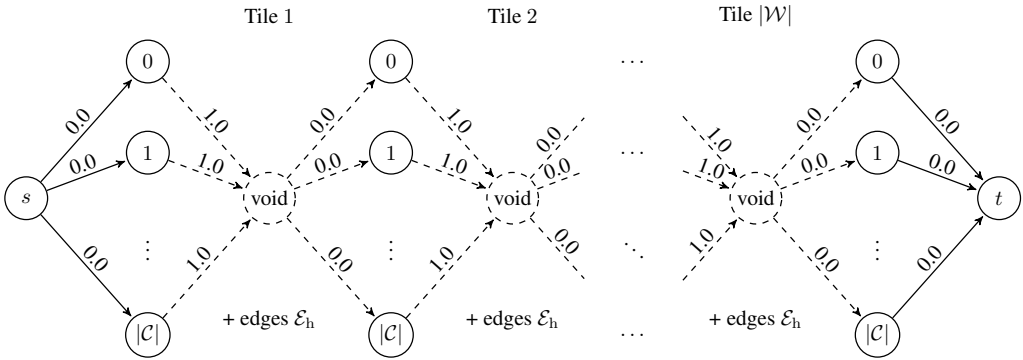

In addition, we assign unitary costs to the edges incoming to the void vertices and zero costs elsewhere. Hence, the $s\rightarrow t$ path cost is equivalent to the number of voids in the row tiling. Furthermore, because the emergent graph is acyclic and single-sourced, the maximum row-cover tiling is found in $\mathcal{O}(\lvert \mathcal{V} \rvert + \lvert \mathcal{E} \rvert)$ time using the DAG-shortest-path algorithm~\cite{korte2006}, where $\mathcal{V}$ denotes the set of the graph vertices and $\mathcal{E}$ the set of the graph edges. In our case, we have
\begin{subequations}
	\begin{align}
	\lvert \mathcal{V} \rvert &= 2 + (\lvert \mathcal{W}\rvert + 1)\lvert \mathcal{C} \rvert + \lvert \mathcal{W}\rvert = 2 + \lvert \mathcal{W}\rvert + \lvert \mathcal{C}\rvert + \lvert \mathcal{W}\rvert\lvert \mathcal{C}\rvert,\\
	\lvert \mathcal{E} \rvert &= 2\lvert \mathcal{C}\rvert + 2 \lvert \mathcal{W} \rvert \lvert \mathcal{C} \rvert + \lvert \mathcal{W} \rvert \lvert \mathcal{T} \rvert.
	\end{align}
\end{subequations}
Thus, the overall asymptotic complexity to generate a~maximum row cover tiling evaluates as
\begin{equation}
\mathcal{O}(\lvert \mathcal{V} \rvert + \lvert \mathcal{E} \rvert) = \mathcal{O}( 2 + \lvert \mathcal{W}\rvert + 3 \lvert \mathcal{C}\rvert + 3 \lvert \mathcal{W} \rvert \lvert \mathcal{C} \rvert + \lvert \mathcal{W} \rvert \lvert \mathcal{T} \rvert) = \mathcal{O}(\lvert \mathcal{W} \rvert \lvert \mathcal{C} \rvert + \lvert \mathcal{W} \rvert \lvert \mathcal{T} \rvert).
\end{equation}
Interestingly, the running time (but not the asymptotic complexity) of the DAG-shortest-path algorithm can be improved by recognizing that the topological order of the graph vertices---which is required for the DAG-shortest-path algorithm---is known from the graph construction method in advance.

Any path with total cost $c_\mathrm{t}$ contains exactly $c_\mathrm{t}$ voids in the row tiling. Because the shortest path algorithm therefore minimizes the number of voids, it generates the maximum row cover as its output. These considerations are summarized below.
\begin{proposition}\label{prop:maximumrowcover}
	The shortest path in the graph in Fig.~\ref{fig:transducerG2} is equivalent to the maximum row cover.
\end{proposition}

\subsection{Tiling consecutive rows}

Assuming already covered rows $i-1$ and $i+1$, e.g., initially by voids, we aim to generate the maximum cover of the $i$-th row. Interestingly, this only requires a~minor modification of the graph in Fig.~\ref{fig:transducerG2}. 

For this, we first check the north-east compatibility for each tile $k \in \mathcal{T}$ placed at $(i,j)$. Notice that the compatibility is never violated when the neighbors are voids. For color mismatch cases, we remove the edges denoting these incompatible tiles from the graph.

Assume that the rows $(i-1)$ and $(i+1)$ are voids. Then, clearly, inappropriate tiles at the $i$-th row may prevent the vertically-adjacent positions to be populated by tiles. To limit the appearance of such introduced voids, we include a~small penalty of $\epsilon = 1/2 (\lvert \mathcal{W} \rvert+1)^{-1}$ to the tiles that admit a~single vertical neighbor only, and $\epsilon = (\lvert \mathcal{W} \rvert+1)^{-1}$ to tiles not admitting any vertical neighbor. Notice that these costs are selected such that, in the worst case, the total penalty due to these void-preventing weights amounts to $\lvert\mathcal{W}\rvert/(\lvert \mathcal{W}\rvert+1)<1$, i.e., the maximum number of tiles is placed even if the void positions forbid any vertical neighbors. Hence, Proposition \ref{prop:maximumrowcover} remains satisfied.

Consequently, we can build a~simple heuristic algorithm,
Alg.~\ref{alg:simple}, that requires $\lvert \mathcal{H} \rvert$ maximum row-cover iterations, rendering the overall complexity to be $\mathcal{O}(\lvert \tilde{\mathcal{A}} \rvert \lvert \mathcal{C} \rvert + \lvert \tilde{\mathcal{A}} \rvert \lvert \mathcal{T} \rvert)$.

\begin{figure*}[!t]
	\vspace*{-7pt}
		\begin{minipage}{\linewidth}%
			\vspace{-5mm}\begin{algorithm}[H]%
				\caption{Simple maximum cover heuristics}%
				\begin{algorithmic}[1]
					\Function{simpleMaximumCoverHeuristics}{$\mathcal{T}$, $\mathcal{A}$}
					\State $\mathfrak{T}$ $\gets$ \texttt{initializeVoidTiling}($\mathcal{A}$)\label{alg:simple:init}
					\State $G_{\mathrm{t},\mathrm{h}}$ $\gets$ \texttt{getTransducerGraph}($\mathcal{T}$)
					\For{row $\gets$ $\{1,\dots,\lvert \mathcal{H} \rvert\}$}
					\State $G_\mathrm{DAG}$ $\gets$ \texttt{constructWeightedDAG}($G_{\mathrm{t},\mathrm{h}}$, $\mathcal{T}$, $\mathfrak{T}$, row)
					\State shortestPath $\gets$ \texttt{solveDAGShortestPathProblem}($G_\mathrm{DAG}$)
					\State $\mathfrak{T}$ $\gets$ \texttt{updateTiling}($\mathfrak{T}$, shortestPath, row)
					\EndFor
					\State \Return $\mathfrak{T}$
					\EndFunction
				\end{algorithmic}
				\label{alg:simple}
			\end{algorithm}
		\end{minipage}
	\vspace*{-7pt}
\end{figure*}

Although Alg.~\ref{alg:simple} usually generates relatively large ratio of the number of placed tiles $\lvert \tilde{\mathcal{B}}_\mathrm{cov}\rvert$ to $\lvert \tilde{\mathcal{A}}\rvert$, it probably lacks a~guaranteed lower bound. Such bounds can, however, be provided by fairly straightforward modifications introduced next.

\subsection{1/2-approximation algorithm for general tile sets}

In this section, we modify Alg.~\ref{alg:simple} to maintain the $1/2$ approximation ratio. We start with the following observation:
\begin{proposition}\label{prop:oddrowcover}
	Consider the maximum row-cover tiling of the \textit{odd} rows of the initially void domain $\mathcal{A}$ given in Section \ref{ssection:sec:maximumrowcover}. Then, $\lvert \tilde{\mathcal{B}}_\mathrm{cov}\rvert \ge 1/2 \lvert \tilde{\mathcal{B}}_{\max \mathrm{cov}} \rvert$.
\end{proposition}
\begin{proof}
	Consider that the maximum row-cover problem alone terminates with $\lvert \tilde{\mathcal{B}}_{\max \mathrm{rowcov}} \rvert$ tiles. Based on the maximum row-cover property in Proposition \ref{prop:maximumrowcover}, none of the rows of $\tilde{\mathcal{A}}$ admit a~tiling by more than $\lvert \tilde{\mathcal{B}}_{\max \mathrm{rowcov}} \rvert$ tiles. Hence, we have $ \lvert\tilde{\mathcal{B}}_\mathrm{cov}\rvert \ge \lceil 1/2 \lvert\mathcal{H}\rvert \rceil \lvert \tilde{\mathcal{B}}_{\max \mathrm{rowcov}}\rvert$ and $\lvert \tilde{\mathcal{B}}_{\max \mathrm{cov}}\rvert \le \lvert\mathcal{H}\rvert \lvert \tilde{\mathcal{B}}_{\max \mathrm{rowcov}}\rvert$, so that $  \lvert \tilde{\mathcal{B}}_\mathrm{cov}\rvert \ge \lceil 1/2\lvert\mathcal{H}\rvert \rceil \lvert \tilde{\mathcal{B}}_{\max \mathrm{rowcov}}\rvert \ge 1/2 \lvert\mathcal{H}\rvert \lvert \tilde{\mathcal{B}}_{\max \mathrm{rowcov}}\rvert \ge 1/2 \lvert \tilde{\mathcal{B}}_{\max \mathrm{cov}} \rvert$, where $\lceil \bullet \rceil$ rounds $\bullet$ to the nearest greater or equal integer.
\end{proof}
To exploit Proposition \ref{prop:oddrowcover} in Alg.~\ref{alg:simple}, we modify the row processing order to $\{1,3,2,5,4\dots\}$. Indeed, then  each odd row contains exactly $\lvert \tilde{\mathcal{B}}_{\max \mathrm{rowcov}}\rvert$ tiles. Nevertheless, covering the $i$-th (odd) row without acknowledging which tiles are placed in the $(i-2)$-th row may result in an unnecessarily empty $(i-1)$-th row. To avoid such situations, we do not check for compatibility with the $(i-1)$-th row voids, but rather we check using the dual transducer graph with the tiles in the $(i-2)$-th row. For each south color code in the $(i-2)$-th row, we find admissible colors (states) in the dual transducer graph as the states reachable by an edge-long path. Indeed, the reached states are exactly the admissible north colors of compatible tiles in the $i$-th row. For the special case of voids in the $(i-2)$-th row, all color codes are assumed to be compatible. Finally, we penalize the incompatibilities with the cost $\epsilon = 1/2 (\lvert \mathcal{W} \rvert+1)^{-1}$ as before. The final algorithm then reads as Alg.~\ref{alg:app0.5}, allowing us to state the following, slightly stronger result:
\begin{figure*}[!b]
	\vspace*{-7pt}
		\begin{minipage}{\linewidth}%
			\vspace{-5mm}\begin{algorithm}[H]%
				\caption{$1/2$-approximation algorithm}%
				\begin{algorithmic}[1]
					\Function{maximumCoverApproximation050}{$\mathcal{T}$, $\mathcal{A}$}
					\State $\mathfrak{T}$ $\gets$ \texttt{initializeVoidTiling}($\mathcal{A}$)
					\State $G_{\mathrm{t},\mathrm{h}}$, $G_{\mathrm{t},\mathrm{v}}$ $\gets$ \texttt{getTransducerGraphs}($\mathcal{T}$)
					\For{row $\gets$ $\{1,3,2,5,4,\dots\}$}
					\If{row even}
					\State $G_\mathrm{DAG}$ $\gets$ \texttt{constructWeightedDAG}($G_{\mathrm{t},\mathrm{h}}$, $\mathcal{T}$, $\mathfrak{T}$, row)
					\Else
					\State $G_\mathrm{DAG}$ $\gets$ \texttt{constructWeightedDAGFromDTransducer}($G_{\mathrm{t},\mathrm{h}}$, $G_{\mathrm{t},\mathrm{v}}$, $\mathfrak{T}$, row, $1$)
					\EndIf
					\State shortestPath $\gets$ \texttt{solveDAGShortestPathProblem}($G_\mathrm{DAG}$)
					\State $\mathfrak{T}$ $\gets$ \texttt{updateTiling}($\mathfrak{T}$, shortestPath, row)
					\EndFor
					\State \Return $\mathfrak{T}$
					\EndFunction
				\end{algorithmic}
				\label{alg:app0.5}
			\end{algorithm}
		\end{minipage}
	\vspace*{-7pt}
\end{figure*}
\begin{proposition}\label{prop:0.5a}
	Assume a~tile set $\mathcal{T}$ with the longest path in its transducer graph $G_{\mathrm{t},\mathrm{h}}$ of at least $2$. Then, Alg.~\ref{alg:app0.5} terminates with $\lvert \tilde{\mathcal{B}}_\mathrm{cov}\rvert \ge 1/2 \lvert \tilde{\mathcal{A}}\rvert$.
\end{proposition}
\begin{proof}
	When $\lvert\tilde{\mathcal{B}}_{\max \mathrm{rowcov}} \rvert = \lvert \mathcal{W}\rvert$, the proof follows directly from Proposition \ref{prop:oddrowcover}. For the other cases, the odd rows must contain $\lvert \tilde{\mathcal{B}}_{\max \mathrm{rowcov}} \rvert$ tiles due to Proposition \ref{prop:maximumrowcover}. Because these row-covers are maximal, the sequence of consecutive voids in these rows cannot exceed two, as we could have placed an additional tile otherwise, contradicting with the maximum row-cover property. Moreover, without loss of generality, the cost of the shortest path in the $i$-th row is at most $\lvert\tilde{\mathcal{B}}_{\max \mathrm{rowcov}} \rvert + (\lvert \mathcal{W}\rvert - \lvert\tilde{\mathcal{B}}_{\max \mathrm{rowcov}} \rvert) \epsilon$, which occurs when the $(i-2)$-th and $i$-th row have the same tile-void patterns. Because the longest void sequence is at most two and the longest path in $G_{\mathrm{t},\mathrm{h}}$ is at least two, we can always place tiles to the north of the voids of the $i$-th row.
\end{proof}

\subsection{2/3-approximation algorithm for tilesets with cyclic transducers}

Another improvement in the approximation factor of Alg.~\ref{alg:app0.5} is possible for tile sets with all the states in the transducer graphs $G_{\mathrm{t},\mathrm{h}}$ and $G_{\mathrm{t},\mathrm{v}}$ being in at least one graph cycle. Notice that this situation occurs for all tile sets that tile the infinite plane.

To this goal, we  modify the assignment of costs to graph, and the row processing order to $\{1,4,3,2,3,7,6,5,6,\dots\}$. We begin with (i) tiling the maximum row-cover of the first row. Then, we (ii) find the maximum row-cover of the $4$th row such that we penalize possible incompatibilities with the first row based on the dual transducer graph by $\epsilon$. The step (iii) encompasses finding a~cover of the $3$rd row with penalized incompatibilities with the first row and enforced voids at even positions. Finally, we find the maximum covers for rows $2$ and $3$. We repeat the procedure for the row numbers iteratively increased by $3$, see Alg.~\ref{alg:3}.
\begin{figure*}[!b]
	\vspace*{-7pt}
		\begin{minipage}{\linewidth}%
			\vspace{-5mm}\begin{algorithm}[H]%
				\caption{$2/3$-approximation algorithm}%
				\begin{algorithmic}[1]
					\Function{maximumCoverApproximation067}{$\mathcal{T}$, $\mathcal{A}$}
					\State $\mathfrak{T}$ $\gets$ \texttt{initializeVoidTiling}($\mathcal{A}$)
					\State $G_{\mathrm{t},\mathrm{h}}$, $G_{\mathrm{t},\mathrm{v}}$ $\gets$ \texttt{getTransducerGraphs}($\mathcal{T}$)
					\State \texttt{setRowsNotVisited}()
					\For{row $\gets$ $\{1,4,3,2,3,7,6,5,6\dots\}$}
					\If{mod(row$-1$,3)$==0$}
					\State $G_\mathrm{DAG}$ $\gets$ \texttt{constructWeightedDAGFromDTransducer}($G_{\mathrm{t},\mathrm{h}}$, $G_{\mathrm{t},\mathrm{v}}$, $\mathfrak{T}$, row, $2$)
					\ElsIf{mod(row$-2$,3)$==0$}
					\If{\texttt{rowVisited}(row)==false}
					\State $G_\mathrm{DAG}$ $\gets$ \texttt{constructWeightedDAGFromDTransducer}($G_{\mathrm{t},\mathrm{h}}$, $G_{\mathrm{t},\mathrm{v}}$, $\mathfrak{T}$, row, $1$)
					\State $G_\mathrm{DAG}$ $\gets$ \texttt{removeTilesAtEvenPositions}($G_\mathrm{DAG}$)
					\Else
					\State $G_\mathrm{DAG}$ $\gets$ \texttt{constructWeightedDAG}($G_{\mathrm{t},\mathrm{h}}$, $\mathcal{T}$, $\mathfrak{T}$, row)
					\EndIf
					\Else
					\State $G_\mathrm{DAG}$ $\gets$ \texttt{constructWeightedDAG}($G_{\mathrm{t},\mathrm{h}}$, $\mathcal{T}$, $\mathfrak{T}$, row)
					\EndIf
					\State shortestPath $\gets$ \texttt{solveDAGShortestPathProblem}($G_\mathrm{DAG}$)
					\State $\mathfrak{T}$ $\gets$ \texttt{updateTiling}($\mathfrak{T}$, shortestPath, row)
					\State \texttt{setRowVisited}(row)
					\EndFor
					\State \Return $\mathfrak{T}$
					\EndFunction
				\end{algorithmic}
				\label{alg:3}
			\end{algorithm}
		\end{minipage}
	\vspace*{-7pt}
\end{figure*}
Then, we can make the following statement:

\begin{lemma}\label{prop:23}
	Consider that all states in the transducer graphs $G_{\mathrm{t},h}$ and $G_{\mathrm{t},\mathrm{v}}$ are in at least one graph cycle. Then, Alg.~\ref{alg:3} terminates with at least $\frac{2}{3} \lvert \tilde{\mathcal{A}}\rvert$ placed tiles.
	\begin{proof}
		Since the tile set allows for valid tiling of the row, the $\{1,4,\;\dots\}$ rows are occupied by exactly $\lvert \mathcal{W} \rvert$ tiles. The $\{3,6,\;\dots\}$ rows are then populated by at least $1/2\lvert \mathcal{W}\rvert$ tiles because each tile from rows $\{4,7,\dots\}$ admits a~vertical neighbor. Finally, the $\{2,5,\;\dots\}$ rows contain at least the complement of the number of tiles used in the preceding row, because the tiles in the $\{1,4,6,\dots\}$ row admit a~south neighbor. Depending on the number of rows, the algorithm places at least
		\begin{equation}
		\lvert \tilde{\mathcal{B}}_{\mathrm{cov}} \rvert \ge \min\{\lvert \tilde{\mathcal{A}} \rvert, \frac{3}{4}\lvert \tilde{\mathcal{A}} \rvert, \frac{2}{3}\lvert \tilde{\mathcal{A}} \rvert, \frac{3}{4}\lvert \tilde{\mathcal{A}} \rvert,\frac{7}{10}\lvert \tilde{\mathcal{A}} \rvert,\frac{2}{3}\lvert \tilde{\mathcal{A}} \rvert,\dots \} = \frac{2}{3} \lvert \tilde{\mathcal{A}} \rvert
		\end{equation}
		tiles.
	\end{proof}
\end{lemma}

\subsection{Iterative improvements}

Similarly to finding the maximum row covers, we can search for the maximum cover of columns. When combining these two methods, we end up with our final algorithm that has the $\mathcal{O}(\lvert \tilde{\mathcal{A}} \rvert^2 \lvert \mathcal{C} \rvert + \lvert \tilde{\mathcal{A}} \rvert^2 \lvert \mathcal{T} \rvert + \lvert \mathcal{C} \rvert^2)$ complexity and provides the approximation ratios adjustable by algorithm choice (Algs.~\ref{alg:simple}, \ref{alg:app0.5} or \ref{alg:3}) at line \ref{alg:initialcover} of Alg.~\ref{alg:improved}.

\begin{figure*}[!b]
	\vspace*{-7pt}
		\begin{minipage}{\linewidth}%
			\vspace{-5mm}\begin{algorithm}[H]%
				\caption{Final maximum cover heuristics}%
				\begin{algorithmic}[1]
					\Function{finalMaximumCoverHeuristics}{$\mathcal{T}$, $\mathcal{A}$}
					\State $\mathfrak{T}$ $\gets$ \texttt{generateInitialCover}($\mathcal{T}$,$\mathcal{A}$)\label{alg:initialcover}
					\State $G_{\mathrm{t},\mathrm{h}}$, $G_{\mathrm{t},\mathrm{v}}$ $\gets$ \texttt{getTransducerGraphs}($\mathcal{T}$)
					\State numVoidsOld $\gets$ $\infty$
					\State method $\gets$ ``columns''
					\While{numVoidsOld$-$\texttt{getNumVoids}($\mathfrak{T}$) $> 0$}
					\State numVoidsOld $\gets$ \texttt{getNumVoids}($\mathfrak{T}$)
					\If{method$==$``rows''}
					\For{row $\gets\{1,\dots,\lvert \mathcal{H}\rvert\}$}
					\State $G_\mathrm{DAG}$ $\gets$ \texttt{constructSimpleDAG}($G_{\mathrm{t},\mathrm{h}}$, $\mathfrak{T}$, row)
					\State shortestPath $\gets$ \texttt{solveDAGShortestPathProblem}($G_\mathrm{DAG}$)
					\State $\mathfrak{T}$ $\gets$ \texttt{updateTiling}($\mathfrak{T}$, shortestPath, row)
					\State method $\gets$ ``columns''
					\EndFor
					\Else
					\For{column $\gets\{1,\dots,\lvert \mathcal{W}\rvert\}$}
					\State $G_\mathrm{DAG}$ $\gets$ \texttt{constructSimpleDAG}($G_{\mathrm{t},\mathrm{v}}$, $\mathfrak{T}$, column)
					\State shortestPath $\gets$ \texttt{solveDAGShortestPathProblem}($G_\mathrm{DAG}$)
					\State $\mathfrak{T}$ $\gets$ \texttt{updateTiling}($\mathfrak{T}$, shortestPath, column)
					\State method $\gets$ ``rows''
					\EndFor
					\EndIf
					%\State improvement $\gets$ numVoidsOld $-$ \texttt{getNumVoids}($\mathfrak{T}$)
					\EndWhile
					\State \Return $\mathfrak{T}$
					\EndFunction
				\end{algorithmic}
				\label{alg:improved}
			\end{algorithm}
		\end{minipage}
	\vspace*{-7pt}
\end{figure*}

\begin{proposition}
	Alg.~\ref{alg:improved} runs in a~polynomial time and terminates in a~finite number of steps.
	\begin{proof}
		We have already shown that finding a~maximum row-cover has $\mathcal{O}(\lvert \mathcal{W} \rvert \lvert \mathcal{C} \rvert + \lvert \mathcal{W} \rvert \lvert \mathcal{T} \rvert)$ complexity. Further, finding the $2$-long paths in the transducer graph possesses the $\lvert \mathcal{C}\rvert^2$ complexity and can be run only once prior to the algorithm main loop. Altogether, Alg.~\ref{alg:3} requires at most $4/3\lvert \mathcal{H}\rvert$ inner iterations so that we have the $\mathcal{O}(\lvert \tilde{\mathcal{A}} \rvert \lvert \mathcal{C} \rvert + \lvert \tilde{\mathcal{A}} \rvert \lvert \mathcal{T} \rvert +  \lvert \mathcal{C} \rvert^2)$ overall complexity.
		
		Regardless of the method at line \ref{alg:initialcover} of Alg.~\ref{alg:improved}, the improving loop runs at most $\lvert \tilde{\mathcal{A}} \rvert$ times. Consequently, the algorithm is finite and possesses the $\mathcal{O}(\lvert \tilde{\mathcal{A}} \rvert^2 \lvert \mathcal{C} \rvert + \lvert \tilde{\mathcal{A}} \rvert^2 \lvert \mathcal{T} \rvert + \lvert\mathcal{C}\rvert^2)$ complexity.
	\end{proof}
\end{proposition}

\section{Results}\label{sec:results}

Having developed several exact and heuristic methods, this section is devoted to their numerical examination. We begin with assessing the performance of the integer programming formulations in Section \ref{ssection:sec:experiments}. Then, in Section \ref{ssection:sec:exp_heuristics}, we also relate these results to the outputs of the heuristic algorithms.

Extensions of the integer programs are investigated in subsequent sections. First, we demonstrate the usefulness of the packing constraint by comparing the efficiency of the solution to the tile-packing problem using our method with the times reported by Lagae and Dutr\'e~\cite{lagae2007}, Section \ref{ssection:sec:tilepacking}. Subsequently, we also present two unexpected discoveries revealed when testing formulations: the Knuth~\cite{knuth1968} tile set contains a~tile unusable in infinite tilings, Section \ref{ssection:sec:knuth}, and the Lagae \textit{et al.}~\cite{lagae2006report} tile set of $44$ corner tiles lacks aperiodicity, Section \ref{ssection:sec:aperiodic}.

We implemented all the methods described above in C++. As the integer programming solver, we used the state-of-the-art optimizer Gurobi 9.5.0 \cite{gurobi} dynamically linked to the compiled binary. Numerical tests were evaluated on a~personal laptop running the Ubuntu $18.04$ operating system equipped with $24$ GB of RAM and Intel$^\text{\textregistered}$ Core$^\text{\textregistered}$ i5-8350U CPU clocked at 1.70GHz.

\subsection{Integer programming formulations}\sslabel{sec:experiments}

In this section, we investigate the performance of all integer programming formulations from Section \ref{sec:intprog}, i.e., the decision program \eqref{eq:binaryFeasibility}, the maximum rectangular tiling \eqref{eq:maxRectangle}, the maximum cover \eqref{eq:covering}, and the maximum adjacency constraint satisfaction problem \eqref{eq:maxCSP}.

\begin{figure}[!t]
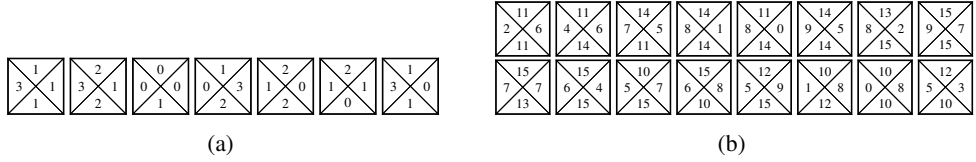

	\vspace*{-7pt}
	\begin{subfigure}[b]{0.475\linewidth}
		\centering
		\resizebox{0.9\linewidth}{!}{
			\WangTile{1}{3}{1}{1}
			\WangTile{2}{3}{2}{1}
			\WangTile{0}{0}{1}{0}
			\WangTile{1}{0}{2}{3}
			\WangTile{2}{1}{2}{0}
			\WangTile{2}{1}{0}{1}
			\WangTile{1}{3}{1}{0}}
		\caption{}
	\end{subfigure}%
	\begin{subfigure}[b]{0.525\linewidth}
		\centering
		\resizebox{0.9\linewidth}{!}{
			\WangTile{11}{2}{11}{6}
			\WangTile{11}{4}{14}{6}
			\WangTile{14}{7}{11}{5}
			\WangTile{14}{8}{14}{1}
			\WangTile{11}{8}{14}{0}
			\WangTile{14}{9}{14}{5}
			\WangTile{13}{8}{15}{2}
			\WangTile{15}{9}{15}{7}}\\
		\resizebox{0.9\linewidth}{!}{
			\WangTile{15}{7}{13}{7}
			\WangTile{15}{6}{15}{4}
			\WangTile{10}{5}{15}{7}
			\WangTile{15}{6}{10}{8}
			\WangTile{12}{5}{15}{9}
			\WangTile{10}{1}{12}{8}
			\WangTile{10}{0}{10}{8}
			\WangTile{12}{5}{10}{3}}
		\caption{}
	\end{subfigure}
	\vspace{-5pt}
	\caption{New tile sets (a) Finite1 of $7$ tiles over $4$ colors, and (b) Finite2 of $16$ tiles over $16$ colors used in our algorithmic tests.}
	\label{fig:finite}
	\vspace*{-7pt}
\end{figure}

\begin{table}[!b]
	\vspace*{-7pt}
	\centering
	\scriptsize
	\setlength{\tabcolsep}{3pt}
	\begin{tabular}{l|c|rr|rr|rr|rr}
		\multirow{2}{6em}{\textbf{Tile set}} & \multirow{2}{2em}{\textbf{Size}} & \multicolumn{2}{c|}{\textbf{Dec. prog.} \eqref{eq:binaryFeasibility}} & \multicolumn{2}{c|}{\textbf{Max. rect.} \eqref{eq:maxRectangle}} & \multicolumn{2}{c|}{\textbf{Max. cov.} \eqref{eq:covering}} & \multicolumn{2}{c}{\textbf{Max. CSP} \eqref{eq:maxCSP}} \\
		& & \textit{Time [s]} & \textit{Objective} & \textit{Time [s]} & \textit{Objective} & \textit{Time [s]} & \textit{Objective} & \textit{Time [s]} & \textit{Objective} \\
		\hline
		\multirow{3}{6em}{Aperiodic1 ($11/4$) \cite{jeandel2021}} & $20\times20$  & $0.111$ & $0$ & $129.897$ & $400$ & $300.053$ & $^*398$ & $300.056$ & $^*745$\\
		& $25\times25$ & $90.810$ & $0$ & $300.070$ & $^*150$ & $300.083$ & $^*606$ & $300.071$ & $^*1135$ \\
		& $30\times30$ & $300.069$ & $^*$infeasible & $300.084$ & $^*150$ & $300.097$ & $^*861$ & $300.089$ & $^*1628$ \\
		\hline
		\multirow{3}{6em}{Aperiodic2 ($13/5$) \cite{culik1996}} & $20\times20$& $0.114$ & $0$ & $145.300$ & $400$ & $300.055$ & $^*399$ & $300.059$ & $^*742$\\
		& $25\times25$ & $178.337$ & $0$ & $300.070$ & $^*125$ & $300.082$ & $^*612$ & $300.078$ & $^*1184$\\
		& $30\times30$ & $300.069$ & $^*$infeasible & $300.089$ & $^*60$ & $300.098$ & $^*876$ & $300.111$ & $^*1655$\\
		\hline
		\multirow{3}{6em}{Aperiodic3 ($14/6$) \cite{kari1996}} & $20\times20$ & $275.339$ & $0$ & $181.171$ & $400$ & $300.058$ & $^*397$ & $300.058$ & $^*752$\\
		& $25\times25$ & $300.057$ & $^*$infeasible & $300.072$ & $^*100$ & $300.086$ & $^*619$ & $300.086$ & $^*1178$ \\
		& $30\times30$ & $300.073$ & $^*$infeasible & $300.092$ & $^*90$ & $300.107$ & $^*863$ & $300.104$ & $^*1610$\\
		\hline
		\multirow{3}{6em}{Aperiodic4 ($16/6$) \cite{grunbaum1987}} & $20\times20$ & $0.142$ & $0$ & $171.136$ & $400$ & $176.584$ & $400$ & $71.141$ & $760$\\
		& $25\times25$ & $0.196$ & $0$ & $300.063$ & $^*100$ & $300.251$ & $^*577$ & $300.085$ & $^*1030$ \\
		& $30\times30$ & $0.251$ & $0$ & $300.265$ & $^*60$ & $300.132$ & $^*794$ & $300.115$ & $^*1616$\\
		\hline
		\multirow{3}{6em}{Aperiodic5 ($56/12$) \cite{robinson1971}} & $20\times20$ & $0.294$ & $0$ & $300.107$ & $^*20$ & $300.214$ & $^*350$ & $302.442$ & $^*688$ \\
		& $25\times25$ & $0.440$ & $0$ & $300.155$ & $^*25$ & $300.354$ & $^*553$ & $300.197$ & $^*1055$ \\
		& $30\times30$ & $0.648$ & $0$ & $300.228$ & $^*30$ & $300.434$ & $^*795$ & $301.102$ & $^*1569$\\
		\hline
		\multirow{3}{6em}{Stochastic1 ($8/2$) \cite{cohen2003}} & $20\times20$ & $0.066$ & $0$ & $0.101$ & $400$ & $0.046$ & $400$ & $4.195$ & $760$ \\
		& $25\times25$ & $0.091$ & $0$ & $0.125$ & $625$ & $0.089$ & $625$ & $5.598$ & $1200$\\
		& $30\times30$ & $0.116$ & $0$ & $0.225$ & $900$ & $0.110$ & $900$ & $10.021$ & $1740$ \\
		\hline
		\multirow{3}{6em}{Stochastic2 ($16/4$) \cite{lagae2006}} & $20\times20$ & $0.114$ & $0$ & $0.107$ & $400$ & $0.129$ & $400$ & $3.226$ & $760$\\
		& $25\times25$ & $0.141$ & $0$ & $0.175$ & $625$ & $0.210$ & $625$ & $6.118$ & $1200$\\
		& $30\times30$ & $0.183$ & $0$ & $0.217$ & $900$ & $0.283$ & $900$ & $6.846$ & $1740$\\
		\hline
		\multirow{3}{6em}{Periodic1 ($10/4$) \cite{wang1975}} & $20\times20$ & $0.121$ & $0$ & $107.475$ & $400$ & $111.982$ & $400$ & $54.696$ & $760$\\
		& $25\times25$ & $0.153$ & $0$ & $274.813$ & $625$ & $300.066$ & $^*584$ & $224.734$ & $1200$\\
		& $30\times30$ & $0.193$ & $0$ & $300.977$ & $^*81$ & $302.606$ & $^*824$ & $300.087$ & $^*1628$\\
		\hline
		\multirow{3}{6em}{Periodic2 ($30/17$) \cite{nurmi2016}} & $20\times20$ & $0.236$ & $0$ & $109.860$ & $400$ & $252.700$ & $400$ & $88.721$ & $760$\\
		& $25\times25$ & $0.325$ & $0$ & $300.103$ & $^*25$ & $300.222$ & $^*545$ & $300.150$ & $^*1017$\\
		& $30\times30$ & $0.473$ & $0$ & $300.158$ & $^*30$ & $300.300$ & $^*786$ & $300.204$ & $^*1521$\\
		\hline
		\multirow{3}{6em}{Finite1\\($7/4$)} & $20\times20$ & $0.066$ & infeasible & $300.025$ & $^*120$ & $300.051$ & $^*378$ & $300.076$ & $^*725$\\
		& $25\times25$ & $0.086$ & infeasible & $300.038$ & $^*125$ & $300.054$ & $^*585$ & $300.061$ & $^*1108$\\
		& $30\times30$ & $0.105$ & infeasible & $300.046$ & $^*108$ & $300.069$ & $^*826$ & $300.080$ & $^*1628$\\
		\hline
		\multirow{3}{6em}{Finite2 ($16/16$)} & $20\times20$ & $0.100$ & infeasible & $300.273$ & $^*40$ & $300.207$ & $^*326$ & $300.077$ & $^*684$\\
		& $25\times25$ & $0.133$ & infeasible & $300.067$ & $^*50$ & $300.111$ & $^*493$ & $300.094$ & $^*1029$\\
		& $30\times30$ & $0.168$ & infeasible & $300.084$ & $^*30$ & $300.131$ & $^*690$ & $300.128$ & $^*1525$\\
	\end{tabular}
	\caption{Benchmark results. Values marked by an asterisk denote a~premature termination of the integer programming solver.}
	\label{tab:benchmark}
	\vspace*{-7pt}
\end{table}

We are unaware of any standard sets for bounded tiling problems except for the specific, mostly aperiodic tile sets listed in the literature, recall Section \ref{ssection:sec:aptile}. Hence, we consider a~set of benchmark problems consisting of five aperiodic tile sets ($11$ tiles over $4$ colors by Jeandel and Rao~\cite{jeandel2021}, $13$ tiles over $5$ colors  by \v{C}ul\'ik~\cite{culik1996}, $14$ tiles over $6$ colors by Kari~\cite{kari1996}, $16$ tiles over $6$ colors by Ammann \cite{grunbaum1987}, and $56$ tiles over $12$ colors by Robinson~\cite{robinson1971}), two stochastic tile sets introduced in computer graphics ($8$ tiles over $2$ colors by Cohen \textit{et al.}~\cite{cohen2003} and a~set of $16$ tiles over $4$ edge colors by Lagae and Dutr\'e ~\cite{lagae2006}), two periodic tile sets ($10$ tiles over $4$ colors by Wang~\cite{wang1975} and the set of $30$ tiles over $17$ edge colors by Lagae~\textit{et al.}~\cite{lagae2006report} and Nurmi~\cite{nurmi2016}). In addition, in Fig.~\ref{fig:finite}, we introduce two tile sets that do not allow for a~valid tiling of the infinite domain. 

For all these tile sets, we aimed at generating valid tilings sized, respectively, $20\times 20$, $25 \times 25$, and $30 \times 30$. The running time of the Gurobi solver was limited to $300$ seconds for the single-threaded mode.

The results shown in Tab.~\ref{tab:benchmark} illustrate that the performance of the decision program \eqref{eq:binaryFeasibility} surpasses any of the candidate variants. However, it failed to find an existent feasible solution in the time limit four times. In these cases, the output of the optimization problems (\ref{eq:maxRectangle}, \ref{eq:covering}, \ref{eq:maxCSP}) provided at least some output. Interestingly, the decision problem \eqref{eq:binaryFeasibility} also was more efficient in the case of proving that the domain $\lvert \mathcal{A}\rvert$ lacks $\mathcal{T}$-tilability. 

Comparison of the optimization variants hints that the maximum cover \eqref{eq:covering} and the maximum adjacency constraint satisfaction \eqref{eq:maxCSP} problems scale better than the maximum rectangular tiling \eqref{eq:maxRectangle}. Indeed, generating any smaller rectangular domain remains \NP-complete, preventing any polynomial-time approximation algorithm to exist. On the other hand, both the formulations \eqref{eq:covering} and \eqref{eq:maxCSP} admit simple heuristics, recall Section \ref{sec:intprog}, allowing the solver to obtain higher-quality feasible solutions faster.

\subsection{Heuristic algorithms}\sslabel{sec:exp_heuristics}

\begin{table}[!b]
	\vspace*{-7pt}
	\centering
	\scriptsize
	\setlength{\tabcolsep}{3pt}
	\begin{tabular}{l|c|rrrr|rrrr|rrrr}
		\multirow{2}{4em}{\textbf{Tile set}} & \multirow{2}{2em}{\textbf{Size}} & \multicolumn{4}{c|}{\textbf{Alg.~\ref{alg:improved} with Alg.~\ref{alg:simple}}} & \multicolumn{4}{c|}{\textbf{Alg.~\ref{alg:improved} with Alg.~\ref{alg:app0.5}}} & \multicolumn{4}{c}{\textbf{Alg.~\ref{alg:improved} with Alg.~\ref{alg:3}}} \\
		& & \textit{$t$ [s]} & \textit{min} & \textit{avg} & \textit{max} & \textit{$t$ [s]} & \textit{min} & \textit{avg} & \textit{max} & \textit{$t$ [s]} & \textit{min} & \textit{avg} & \textit{max} \\
		\hline
		\multirow{3}{6em}{Aperiodic1 ($11/4$) \cite{jeandel2021}} & $20\times20$ & $0.024$ & $358$ & $\mathbf{368.99}$ & $380$ & $0.056$ & $342$ & $360.22$ & $372$ & $0.029$ & $334$ & $347.60$ & $360$\\
		& $25\times25$ & $0.040$ & $562$ & $\mathbf{575.38}$ & $586$ & $0.101$ & $543$ & $563.67$ & $582$ & $0.046$ & $524$ & $537.68$ & $547$\\
		& $30\times30$ & $0.065$ & $813$ & $\mathbf{829.66}$ & $843$ & $0.160$ & $792$ & $812.33$ & $835$ & $0.080$ & $758$ & $779.91$ & $798$\\
		\hline
		\multirow{3}{6em}{Aperiodic2 ($13/5$) \cite{culik1996}} & $20\times20$ & $0.043$ & $354$ & $\mathbf{369.86}$ & $381$ & $0.054$ & $326$ & $359.86$ & $373$ & $0.051$ & $326$ & $353.21$ & $373$\\
		& $25\times25$ & $0.065$ & $564$ & $\mathbf{577.31}$ & $590$ & $0.099$ & $504$ & $564.92$ & $581$ & $0.118$ & $520$ & $557.49$ & $586$ \\
		& $30\times30$ & $0.123$ & $818$ & $\mathbf{831.60}$ & $847$ & $0.179$ & $800$ & $817.11$ & $837$ & $0.186$ & $760$ & $806.65$ & $838$ \\
		\hline
		\multirow{3}{6em}{Aperiodic3 ($14/6$) \cite{kari1996}} & $20\times20$ & $0.034$ & $362$ & $375.40$ & $386$ & $0.032$ & $353$ & $365.57$ & $381$ & $0.042$ & $355$ & $\mathbf{378.87}$ & $388$\\
		& $25\times25$ & $0.058$ & $564$ & $585.56$ & $604$ & $0.056$ & $548$ & $569.15$ & $597$ & $0.065$ & $562$ & $\mathbf{592.28}$ & $604$\\
		& $30\times30$ & $0.092$ & $813$ & $843.19$ & $857$ & $0.095$ & $799$ & $830.93$ & $860$ & $0.102$ & $827$ & $\mathbf{855.68}$ & $871$\\
		\hline
		\multirow{3}{6em}{Aperiodic4 ($16/6$) \cite{grunbaum1987}} & $20\times20$ & $0.031$ & $351$ & $\mathbf{366.09}$ & $381$ & $0.077$ & $293$ & $333.98$ & $351$ & $0.092$ & $281$ & $339.30$ & $368$\\
		& $25\times25$ & $0.052$ & $555$ & $\mathbf{573.19}$ & $591$ & $0.149$ & $469$ & $524.82$ & $549$ & $0.187$ & $442$ & $533.27$ & $562$\\
		& $30\times30$ & $0.096$ & $795$ & $\mathbf{825.44}$ & $860$ & $0.234$ & $666$ & $758.52$ & $785$ & $0.276$ & $743$ & $773.66$ & $802$\\
		\hline
		\multirow{3}{6em}{Aperiodic5 ($56/12$) \cite{robinson1971}} & $20\times20$ & $0.054$ & $344$ & $\mathbf{360.55}$ & $381$ & $0.149$ & $256$ & $341.48$ & $364$ & $0.171$ & $290$ & $332.83$ & $349$\\
		& $25\times25$ & $0.110$ & $540$ & $\mathbf{563.41}$ & $607$ & $0.289$ & $402$ & $527.68$ & $563$ & $0.472$ & $484$ & $529.63$ & $552$\\
		& $30\times30$ & $0.147$ & $782$ & $\mathbf{811.17}$ & $856$ & $0.473$ & $598$ & $782.99$ & $809$ & $0.601$ & $706$ & $759.48$ & $786$\\
		\hline
		\multirow{3}{6em}{Stochastic1 ($8/2$) \cite{cohen2003}} & $20\times20$ & $0.014$ & $400$ & $\mathbf{400.00}$ & $400$ & $0.012$ & $400$ & $\mathbf{400.00}$ & $400$ & $0.014$ & $400$ & $\mathbf{400.00}$ & $400$\\
		& $25\times25$ & $0.013$ & $625$ & $\mathbf{625.00}$ & $625$ & $0.014$ & $625$ & $\mathbf{625.00}$ & $625$ & $0.021$ & $625$ & $\mathbf{625.00}$ & $625$\\
		& $30\times30$ & $0.016$ & $900$ & $\mathbf{900.00}$ & $900$ & $0.016$ & $900$ & $\mathbf{900.00}$ & $900$ & $0.019$ & $900$ & $\mathbf{900.00}$ & $900$\\
		\hline
		\multirow{3}{6em}{Stochastic2 ($16/4$) \cite{lagae2006}} & $20\times20$ & $0.013$ & $400$ & $\mathbf{400.00}$ & $400$ & $0.015$ & $400$ & $\mathbf{400.00}$ & $400$ & $0.015$ & $400$ & $\mathbf{400.00}$ & $400$\\
		& $25\times25$ & $0.017$ & $625$ & $\mathbf{625.00}$ & $625$ & $0.017$ & $625$ & $\mathbf{625.00}$ & $625$ & $0.019$ & $625$ & $\mathbf{625.00}$ & $625$\\
		& $30\times30$ & $0.025$ & $900$ & $\mathbf{900.00}$ & $900$ & $0.022$ & $900$ & $\mathbf{900.00}$ & $900$ & $0.026$ & $900$ & $\mathbf{900.00}$ & $900$\\
		\hline
		\multirow{3}{6em}{Periodic1 ($10/4$) \cite{wang1975}} & $20\times20$ & $0.026$ & $342$ & $\mathbf{354.82}$ & $374$ & $0.059$ & $325$ & $339.42$ & $353$ & $0.044$ & $323$ & $337.66$ & $346$\\
		& $25\times25$ & $0.043$ & $533$ & $\mathbf{553.13}$ & $573$ & $0.103$ & $507$ & $528.02$ & $546$ & $0.065$ & $512$ & $527.11$ & $537$\\
		& $30\times30$ & $0.065$ & $770$ & $\mathbf{797.26}$ & $830$ & $0.165$ & $745$ & $764.69$ & $784$ & $0.097$ & $729$ & $758.36$ & $770$\\
		\hline
		\multirow{3}{6em}{Periodic2 ($30/17$) \cite{nurmi2016}} & $20\times20$ & $0.092$ & $366$ & $\mathbf{382.23}$ & $400$ & $0.419$ & $259$ & $337.18$ & $368$ & $0.346$ & $281$ & $345.72$ & $369$\\
		& $25\times25$ & $0.153$ & $559$ & $\mathbf{595.04}$ & $620$ & $0.808$ & $503$ & $533.30$ & $568$ & $0.649$ & $537$ & $563.63$ & $575$\\
		& $30\times30$ & $0.237$ & $825$ & $\mathbf{854.86}$ & $887$ & $1.259$ & $724$ & $760.08$ & $802$ & $1.131$ & $768$ & $800.74$ & $822$ \\
		\hline
		\multirow{3}{6em}{Finite1\\($7/4$)} & $20\times20$ & $0.023$ & $353$ & $360.71$ & $369$ & $0.047$ & $352$ & $\mathbf{362.98}$ & $376$ & $0.054$ & $348$ & $357.72$ & $366$\\
		& $25\times25$ & $0.039$ & $548$ & $\mathbf{561.15}$ & $570$ & $0.073$ & $348$ & $559.37$ & $580$ & $0.091$ & $539$ & $554.12$ & $566$\\
		& $30\times30$ & $0.056$ & $789$ & $807.79$ & $819$ & $0.111$ & $789$ & $\mathbf{809.48}$ & $830$ & $0.128$ & $791$ & $804.21$ & $826$\\
		\hline
		\multirow{3}{6em}{Finite2 ($16/16$)} & $20\times20$ & $0.065$ & $334$ & $\mathbf{344.36}$ & $363$ & $0.113$ & $286$ & $331.80$ & $353$ & $0.104$ & $291$ & $328.83$ & $355$ \\
		& $25\times25$ & $0.101$ & $518$ & $536.20$ & $552$ & $0.198$ & $493$ & $529.24$ & $545$ & $0.140$ & $511$ & $\mathbf{544.78}$ & $559$\\
		& $30\times30$ & $0.191$ & $745$ & $\mathbf{770.34}$ & $789$ & $0.393$ & $695$ & $761.45$ & $780$ & $0.382$ & $678$ & $750.89$ & $790$\\
	\end{tabular}
	\caption{Numerical tests of the maximum-cover heuristics, Alg.~\ref{alg:improved}, initialized based on Algs.~\ref{alg:simple}, \ref{alg:app0.5}, and \ref{alg:3}. Best mean runs are highlighted in bold.}
	\label{tab:benchmarkheur}
	\vspace*{-7pt}
\end{table}

Second, we compare the performance of the maximum cover formulation \eqref{eq:covering} solved with the heuristic Alg.~\ref{alg:improved} supplied with three different initial coverings, i.e., based on Algs.~\ref{alg:simple}, \ref{alg:app0.5} and \ref{alg:3}.

Alg.~\ref{alg:improved} ran sequentially. In order to limit the dependence of the heuristic algorithm on the ordering of tiles, we randomized the edge order in the directed acyclic graphs. Thus, we evaluated Alg.~\ref{alg:improved} $100$ times for each of the tested option, and listed the best, worst, and mean results in Tab.~\ref{tab:benchmarkheur}.

From Tab.~\ref{tab:benchmarkheur}, it follows that the initialization with the cover from Alg.~\ref{alg:simple} is the most efficient for the tested tile sets, both in terms of speed and performance. The remaining two initializations seem to be fairly comparable on average. While for Alg.~\ref{alg:simple}, at least $82\%$ of tiles were always placed, only more than $60\%$ followed from Alg.~\ref{alg:app0.5}. Using Alg.~\ref{alg:3}, we obtained at least $70\%$ tile placement.

When comparing Tab.~\ref{tab:benchmark} with Tab.~\ref{tab:benchmarkheur}, a~few patterns emerge. First, the heuristic algorithm always generates valid tilings if (any of) the stochastic tile sets are used. For aperiodic and periodic tile sets, Gurobi required a~considerably longer time to reach feasible solutions of a~similar quality, but usually surpassed the developed algorithms in the time limit of $300$~s. In the case of Alg.~\ref{alg:simple}, it can be seen that the resulting covers are very competitive to the outputs of \eqref{eq:covering} and also obtained in much shorter times.

\begin{figure}[!t]
	\vspace*{-7pt}
	\centering
	\includegraphics[width=0.6\linewidth]{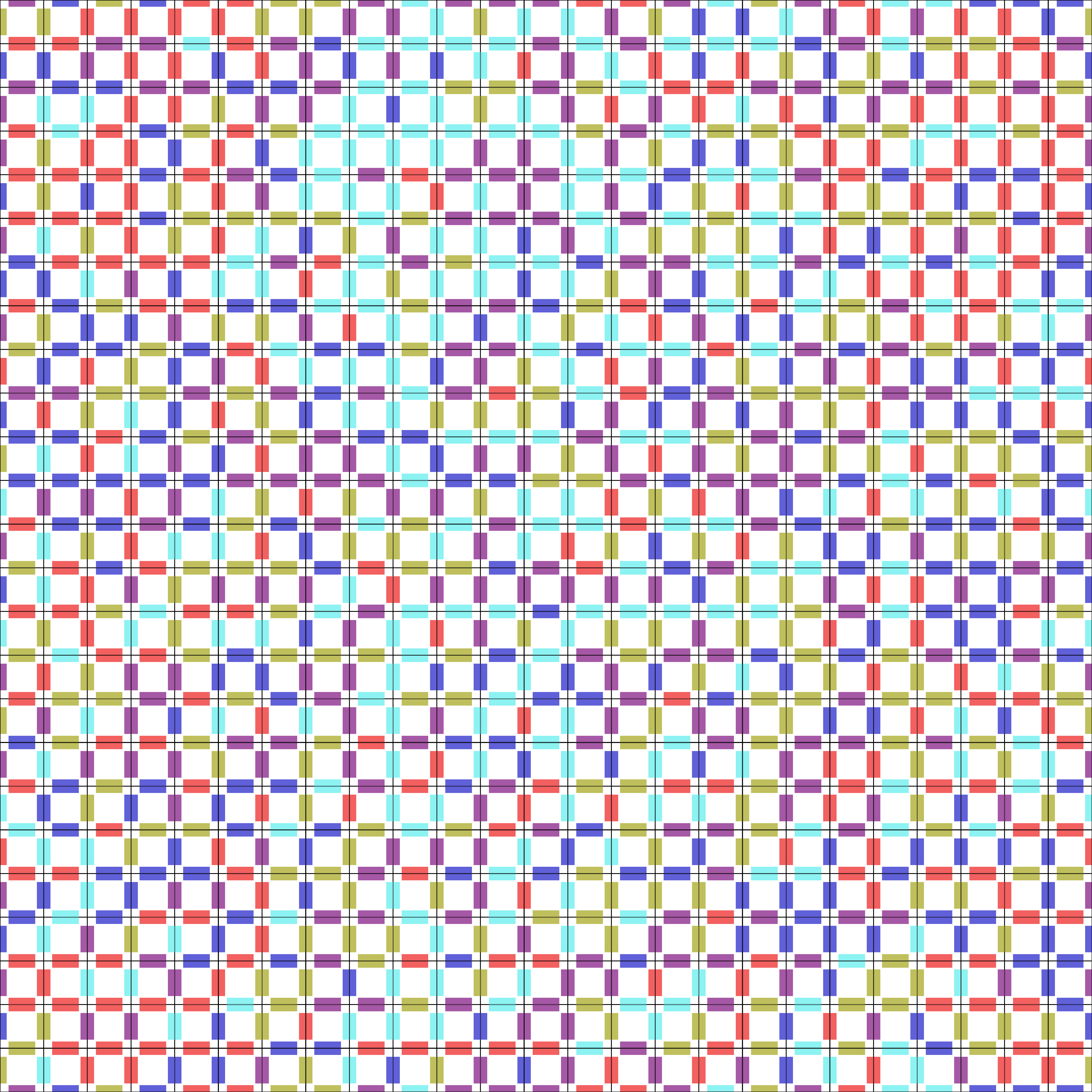}
	\caption{Periodic packing of a~complete set of $625$ tiles over $5$ colors.}
	\vspace*{-7pt}
\end{figure}

\subsection{Periodic tile packing problem}\sslabel{sec:tilepacking}

\begin{table}[!b]
	\vspace*{-7pt}
	\small
	\centering
	\begin{tabular}{lcc}
		\bfseries Tile set & \bfseries Time & \bfseries Time\\
		& (\ref{eq:binaryFeasibility},\ref{eq:periodicfixed},\ref{eq:packing}) & Lagae and Dutr\'e~\cite{lagae2007}\\
		\hline Stochastic edge ($16/2$) & $<1$ sec. & $<1$ sec.\\
		Stochastic edge ($81/3$) & $<1$ sec. & $<1$ sec.\\
		Stochastic edge ($256/4$) & $9$ sec. & $140$ days\\
		Stochastic edge ($625/5$) & $4$ days & -\\
	\end{tabular}
	\caption{Periodic tile packing problem: comparison of core times needed to find a~single feasible solution by integer programming (second column) and by the backtracking method (third column) proposed in Lagae and Dutr\'e~\cite{lagae2007} to find a~feasible solution.}
	\label{tab:packing}
	\vspace*{-7pt}
\end{table}

As the second numerical example, we consider the periodic tile packing problem investigated in computer graphics applications \cite{lagae2007}. Considering a~complete edge tile set, Lagae and Dutr{\'{e}} searched for a~periodic square valid tiling with each tile from the tile set used exactly once. Clearly, such tilings not only contain the entire (textural) information stored in individual tiles but also maintain compatibility with the traditional periodic arrangement.

While Lagae and Dutr\'e~\cite{lagae2007} proposed a~backtracking-based algorithm to generate periodic packings, we rely here on a~solution to the decision program \eqref{eq:binaryFeasibility} supplemented with the packing \eqref{eq:packing} and fixed periodicity \eqref{eq:periodicfixed} constraints. The resulting core times spent in the search for a~single feasible solution (Tab.~\ref{tab:packing}) illustrates the higher effectiveness of our method.

\subsection{Unusable tile in the Knuth tile set}\sslabel{sec:knuth}

One of the oldest aperiodic tile sets, containing $92$ tiles over $26$ colors, is from Knuth~\cite[Exercise~5 in Section 2.3.4.3]{knuth1968}. Generating valid tilings from the Knuth tile set using the decision program \eqref{eq:binaryFeasibility} together with the tile-based boundary conditions, recall Section \ref{sssection:sec:tile_based_boundary}, led to an~unexpected observation that enforced placement of the $\beta U S$ tile makes the program \eqref{eq:binaryFeasibility} infeasible under certain circumstances. 

After a~careful investigation, it indeed turned out that there is not any $2 \times 2$ valid tiling with the $\beta U S$ tile placed at $(2,2)$. Moreover, there is also not any $4 \times 3$ valid tiling with the $\beta U S$ tile placed at $(3,1)$. Thus, using the maximum-cover optimization variant \eqref{eq:covering} and the $\beta U S$ tile enforced at the respective coordinate, there are exactly $31$ optimal solutions with the objective function equal to $3$, and $498$ optimal solutions with the objective function equal to $11$.

Consequently, the $\beta U S$ tile can appear only in the strip of at most $3$ consecutive infinite columns and does not allow for simply-connected valid tilings of the infinite plane. In a~private communication, Knuth confirmed the issue and discovered another $5$ tiles that are \textit{unnecessary} but usable in infinite valid tilings, allowing for a~possible reduction of the tile set to $86$ tiles. For more information, we refer the interested reader to Knuth's discussion about the reduced tile set \cite[Exercise~221 in Section 7.2.2.1]{knuth2018}.

\subsection{Periodicity of the Lagae corner tile set}\sslabel{sec:aperiodic}

Analogously to the Wang tiles, with the connectivity information stored in the edges, Lagae and Dutr\'e~\cite{lagae2006} introduced \textit{corner tiles} with  colored corners. As Wang~\cite{wang1975} noted in 1975, these formalisms are interchangeable if the (infinite) domino problem is considered, because every set of Wang tiles can be represented by sets of corner tiles with greater or equal cardinality~\cite{lagae2006report}. However, corner tiles avoid the so-called corner problem of Wang tiles in computer graphics~\cite{lagae2006}, motivating Lagae \textit{et al.}~\cite{lagae2006report} to develop conversion methods for transforming Wang tiles to corner tiles, and vice versa. A direct product of these conversions are aperiodic tile sets of corner tiles~\cite{lagae2006report}. 

Two of these methods, called horizontal and vertical translations, were used to convert the Ammann set of $16$ Wang tiles over $6$ colors~\cite{grunbaum1987} to the set of $44$ corner tiles over $6$ colors, and the resulting isomorphic corner tile sets were claimed aperiodic~\cite{lagae2006report}. In 2016, Nurmi~\cite{nurmi2016} noticed that, in this set, $14$ tiles are unusable in infinite valid tilings, and reduced the corner tile set to $30$ tiles over $6$ colors. Quite surprisingly, neither Lagae \textit{et. al.} nor Nurmi recognized that the tile set forms a~torus, and is therefore periodic, as we show next.

Instead of developing a~new formulation for another type of tiles, we first notice that corner tiles are actually a~subset of Wang tiles, and therefore every set of corner tiles can be represented by a~set of Wang tiles with the same cardinality, see Appendix~\ref{sec:corner_edge}. For these tiles, we solve the rectangular tiling formulation \eqref{eq:maxRectangle} with periodic boundary conditions \eqref{eq:periodicRectangle} and an objective function to find the smallest tiling \eqref{eq:smallestTiling}. As its output, we receive the optimal value of $6$ and $12$ optimal periodic rectangular tilings of the size $2 \times 3$. Not surprisingly, all these solutions follow from only two periodic patterns shown in Fig.~\ref{fig:lagae_periodic} by translations over the infinite plane.
\begin{figure}[!htbp]
	\vspace*{-7pt}
	\centering
	\hspace*{\fill}%
	\begin{subfigure}{2.04cm}
		\centering
		\resizebox{1.75cm}{!}{\CornerTile{2}{3}{0}{1}\hspace{-1pt}\CornerTile{1}{0}{3}{2}}\\[-2pt]
		\resizebox{1.75cm}{!}{\CornerTile{3}{5}{1}{0}\hspace{-1pt}\CornerTile{0}{1}{5}{3}}\\[-2pt]
		\resizebox{1.75cm}{!}{\CornerTile{5}{2}{1}{1}\hspace{-1pt}\CornerTile{1}{1}{2}{5}}
		\caption{}
	\end{subfigure}%
	\hspace*{\fill}%
	\begin{subfigure}{2.04cm}
		\centering
		\resizebox{1.75cm}{!}{\CornerTile{2}{3}{0}{1}\hspace{-1pt}\CornerTile{1}{0}{3}{2}}\\[-2pt]
		\resizebox{1.75cm}{!}{\CornerTile{3}{4}{0}{0}\hspace{-1pt}\CornerTile{0}{0}{4}{3}}\\[-2pt]
		\resizebox{1.75cm}{!}{\CornerTile{4}{2}{1}{0}\hspace{-1pt}\CornerTile{0}{1}{2}{4}}
		\caption{}
	\end{subfigure}%
	\hspace*{\fill}%
	\vspace{-5pt}
	\caption{Rectangular periodic valid tilings. Translating a~$2\times3$ rectangle over the infinite valid tiling generated from (a) or (b) leads to $6$ different periodic patterns of the same size. Consequently, the tile set allows for $12$ periodic rectangles of the size $2\times 3$.}
	\label{fig:lagae_periodic}
	\vspace*{-7pt}
\end{figure}

Having revealed the smallest periodic solutions, it remains to be shown why the Lagae conversion methods failed. Lagae \textit{et al.}~\cite{lagae2006report} mentioned that their methods lack bijectiveness in general but they assumed it was not the case here. Therefore, we believe it is useful to state the conditions under which the methods are bijective and show that they are not satisfied for the Ammann tile set.

\begin{lemma}\label{prop:translation}
	The horizontal translation method from Lagae \textit{et al.}~\cite{lagae2006report} is bijective iff the dual transducer graph $G_\mathrm{T,v}$ of the input tile set $\mathcal{T}$ does not contain any parallel arcs.
\end{lemma}

\begin{proof}
	The horizontal translation method is formally a~mapping $\mathcal{T} \times \mathcal{T} \mapsto \mathcal{T}_\mathrm{corner}$ that generates $\forall (p,q) \in \mathcal{T}^2: c_p^\mathrm{e} = c_q^\mathrm{w}$ a~corner tile $(c_p^\mathrm{n}, c_p^\mathrm{s}, c_q^\mathrm{s}, c_q^\mathrm{n})$. To be bijective, the cardinality of the output needs to be equal to the cardinality of the input, and the mapping has to produce unique output for each input. Consequently, all the tiles $p \in \mathcal{T}$ in the original tile set must be uniquely determined by $c_p^\mathrm{n}$ and $c_p^\mathrm{s}$, as the color codes of the vertical edges of $\mathcal{T}$ are avoided in the construction of $\mathcal{T}_\mathrm{corner}$. 
	
	Let us now consider that the dual transducer graph contains a~parallel arc connecting the state $c^\mathrm{n}$ with $c^\mathrm{s}$. Then, there may exist two tiles colored by $(c^\mathrm{n}, c_p^\mathrm{w}, c^\mathrm{s}, c_p^\mathrm{e})$ and $(c^\mathrm{n}, c_q^\mathrm{w}, c^\mathrm{s}, c_q^\mathrm{e})$ that are indistinguishable in $\mathcal{T}_\mathrm{corner}$, which contradicts the bijection. For the other option, if the transducer graph does not contain any parallel arcs, then each $c_q^\mathrm{n}$, $c_q^\mathrm{s}$ identifies with a~single arc labeled by $c_q^\mathrm{w} \vert c_q^\mathrm{e}$, i.e., with a~single tile, which completes the proof.
\end{proof}

Rotating the tile set by $90$ degrees, the arguments in Lemma \ref{prop:translation} provide us with the conditions for the bijectiveness of the vertical translation method:

\begin{lemma}
	The vertical translation method of Lagae \textit{et al.}~\cite{lagae2006report} is bijective iff the transducer graph $G_\mathrm{T,h}$ of the input tile set $\mathcal{T}$ does not contain any parallel arcs.
\end{lemma}

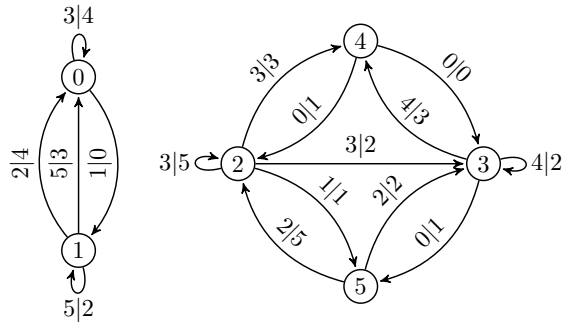
\begin{figure}[!b]
	\vspace*{-7pt}
	\centering
	\resizebox{0.55\linewidth}{!}{\begin{tikzpicture}[>=stealth',shorten >=1pt,auto,node distance=2cm, semithick,every node/.style={inner sep=2pt}, every state/.style={minimum size=0pt,inner sep=2pt}]
		\node (0) [state] {$0$};
		\node (1) [state, below=of 0] {$1$};
		\path[->] (0) edge [bend left=40, above, sloped, rotate=180] node(b) {$1\vert0$} (1);
		\path[->] (0)  edge [loop above, sloped] node {$3\vert4$} (0);
		\path[->] (1)  edge [loop below, sloped] node {$5\vert2$} (1);
		\path[->] (1) edge [bend left=40, above, sloped]  node {$2\vert4$} (0);
		\path[->] (1) edge node [above, sloped] {$5\vert3$} (0);
		
		\node (2) [state, right=of b.north] {$2$};
		\node (5) [state, below right=of 2] {$5$};
		\node (4) [state, above right=of 2] {$4$};
		\node (3) [state, below right=of 4] {$3$};
		\path[->] (2) edge node [above, sloped] {$3\vert2$} (3);
		\path[->] (2) edge [loop left] node {$3\vert5$} (2);
		\path[->] (2) edge [bend left, sloped, above] node {$3\vert3$} (4);
		\path[->] (2) edge [bend left, sloped, above] node {$1\vert1$} (5);
		\path[->] (5) edge [bend left, sloped, above] node {$2\vert5$} (2);
		\path[->] (3) edge [bend left, sloped, above] node {$4\vert3$} (4);
		\path[->] (3) edge [bend left, sloped, above] node {$0\vert1$} (5);
		\path[->] (5) edge [bend left, sloped, above] node {$2\vert2$} (3);
		\path[->] (3) edge [loop right] node {$4\vert2$} (3);
		\path[->] (4) edge [bend left, sloped, above] node {$0\vert1$} (2);
		\path[->] (4) edge [bend left, sloped, above] node {$0\vert0$} (3);
		\end{tikzpicture}}
	\caption{Transducer graph of the Ammann set of $16$ Wang tiles over $6$ colors.}
	\label{fig:transducer}
	\vspace*{-20pt}
\end{figure}

For the Ammann tile set, we obtain the transducer graph $G_\mathrm{T,h} = G_\mathrm{T,v}$ shown in Fig.~\ref{fig:transducer}. Clearly, there exist parallel arcs $1 \rightarrow 0$. Moreover, using the same approach, we can show that the horizontal translation method also fails for the Robinson tile set of $24$ tiles over $24$ colors~\cite{grunbaum1987}, contrary to the claims in \cite{lagae2006report}, and the corresponding corner tile set is also periodic.

\section{Conclusions}\label{sec:wangconc}

In this contribution, we investigated methods generating bounded Wang tilings for arbitrary tile sets. To this goal, we developed an \NP-complete binary linear programming formulation \eqref{eq:binaryFeasibility}, as well as its \NP-hard optimization variants relaxing some of the initial assumptions: tilability of the entire rectangular domain leading to the maximum rectangular tiling formulation \eqref{eq:maxRectangle}, simple-connectedness to the maximum-cover program \eqref{eq:covering}, and tiling validity to the maximum constraint satisfaction problem \eqref{eq:maxCSP}. In addition, we supplemented these formulations with extensions enabling control over individual tiles and their colors, including a~tile-packing constraint to enable generation of periodic tile packings, and introduced the variable-sized periodic constraints to facilitate computation of the smallest periodic patterns.

Motivated by the \NP-hardness of all optimization formulations, we developed simple yet efficient heuristic algorithms for the maximum-cover variant \eqref{eq:covering}. These algorithms rely on the fact that generating the maximum row cover is equivalent to finding the shortest path in a~directed acyclic graph. Moreover, well-chosen costs of the graph edges also maintain color matches with the neighboring rows. Thus, in the simplest case, a~heuristic solution follows from a~sequential generation of row-cover tilings. Moreover, with simple modifications to the row processing order, we showed how to provide a~$1/2$ approximation factor for general tile sets and a~$2/3$ guarantee for tile sets whose transducer graphs are cyclic.

We illustrated the effectiveness of these methods on a~collection of $11$ tile sets. Generating tilings sized $20\times20$, $25\times25$, and $30\times30$ revealed that the decision program~\eqref{eq:binaryFeasibility} is the most efficient for our test problems. However, when a~time limit is imposed or if the tile set does not allow for valid tiling of the entire domain, then the maximum cover \eqref{eq:covering} and maximum adjacency constraint satisfaction problems \eqref{eq:maxCSP} appear to be similarly efficient. The remaining formulation---maximum rectangular tiling \eqref{eq:maxRectangle}---exhibits the worst scalability.

The usefulness of integer programming extensions was demonstrated by means of three illustrative problems: showing a~better solution efficiency to the tile packing problem than the Lagae and Dutr\'e~\cite{lagae2007} backtracking approach, revealing an unusable tile in the Knuth~\cite{knuth1968} tile set, and proving that the Lagae \textit{et al.}~\cite{lagae2006report} tile set of corner tiles lacks aperiodicity. For the latter case, we also included an explanation for why the tile set construction method failed.

Furthermore, we tested the performance of the heuristic algorithms against the Gurobi solver tackling integer problem \eqref{eq:covering} for $300$~s. Testing revealed that the simplest setup of Alg.~\ref{alg:improved} initialized with the cover generated by Alg.~\ref{alg:simple} produces the best results, obtained faster and (in most cases) competitive with the outputs of Gurobi. Somewhat surprisingly, the variants with guaranteed lower bounds exhibited slightly worse performance on average.

Having summarized our contributions, we believe that this work has not only introduced new methods that can possibly be applied to materials engineering, but also a~simple and quite extensible framework to verify theoretical results on Wang tilings.

\appendix
	\section{Corner tiles represented as Wang tiles}\label{sec:corner_edge}
	Each corner tile in the corner tile set is defined by a~quadruple of color codes $(c_k^\mathrm{nw},c_k^\mathrm{sw},c_k^\mathrm{se},c_k^\mathrm{ne})$, with $c_k^\mathrm{nw}$, $c_k^\mathrm{sw}$, $c_k^\mathrm{se}$, and $c_k^\mathrm{ne}$ denoting the colors of the northwest, southwest, southeast, and northeast corner of the $k$-th tile. Similarly to Wang tiles, corner tiles are assembled such that the color codes at the adjoining corners match. 
	
	This is, however, also maintained if we denote their edges by labels in the form of tuples of the corner codes, $(c_k^\mathrm{nw}, c_k^\mathrm{ne})$, $(c_k^\mathrm{nw}, c_k^\mathrm{sw})$, $(c_k^\mathrm{sw}, c_k^\mathrm{se})$, and $(c_k^\mathrm{ne}, c_k^\mathrm{se})$, each of which denotes a~single edge label of the north, west, south, and east edge, respectively. Consequently, we can compute unique color codes as

	\begin{subequations}
		\begin{minipage}{0.48\linewidth}
		\centering
		\begin{equation}
		c_k^\mathrm{n} = c_k^\mathrm{nw} + c_k^\mathrm{ne} n_\mathrm{vc},
		\end{equation}
		\begin{equation}
		c_k^\mathrm{w} = c_k^\mathrm{nw} + c_k^\mathrm{sw} n_\mathrm{vc},
		\end{equation}
		\end{minipage}%
		\hfill\begin{minipage}{0.48\linewidth}
		\centering
		\begin{equation}
		c_k^\mathrm{s} = c_k^\mathrm{sw} + c_k^\mathrm{se} n_\mathrm{vc},
		\end{equation}
		\begin{equation}
		c_k^\mathrm{e} = c_k^\mathrm{ne} + c_k^\mathrm{se} n_\mathrm{vc},
		\end{equation}
		\end{minipage}%
		\vspace{4mm}
	\end{subequations}
	
	\noindent where $n_\mathrm{vc}$ stands for the number of colors used in the corner tile set. Graphical illustration of the corner-edge tile equivalence is shown in Fig.~\ref{fig:edgecorner}.
	
	\begin{figure}[!t]
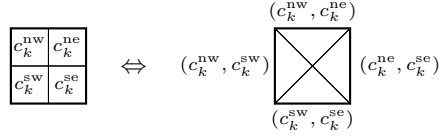

		\vspace*{-7pt}
		\centering
		\CornerTile{$c_k^\mathrm{nw}$}{$c_k^\mathrm{sw}$}{$c_k^\mathrm{se}$}{$c_k^\mathrm{ne}$}
		$\quad\Leftrightarrow\quad$
		\WangTile{\begin{tabular}{c}$(c_k^\mathrm{nw},c_k^\mathrm{ne})$\\{}\\{}\\{}\end{tabular}}{$(c_k^\mathrm{nw}, c_k^\mathrm{sw})\qquad\qquad\qquad\quad$}{\begin{tabular}{c}{}\\{}\\{}\\$(c_k^\mathrm{sw},c_k^\mathrm{se})$\end{tabular}}{$\qquad\qquad\qquad\quad(c_k^\mathrm{ne}, c_k^\mathrm{se})$}
		\caption{A corner tile expressed using edge formalism.}
		\label{fig:edgecorner}
		\vspace*{-7pt}
	\end{figure}

\enlargethispage{20pt}

\paragraph{Data accessibility} Source code available at: \url{https://gitlab.com/tyburec/tilopt}.

\paragraph{Authors' Contributions} M.T.: conceptualization, methodology, software, writing -- original draft, J.Z.: methodology,  writing -- review \& editing, supervision, funding acquisition.

\paragraph{Competing Interests} The authors declare there is no competing interest.

\paragraph{Funding} This research was funded by the Czech Science Foundation, project No. 19-26143X.

\paragraph{Acknowledgments} We thank Stephanie Krueger for proofreading the initial draft of this manuscript.

%%%%%%%%%% Insert bibliography here %%%%%%%%%%%%%%

\vskip2pc

\bibliographystyle{RS}
\bibliography{liter}

\end{document}